\documentclass[a4paper]{article}

\usepackage{amsfonts,amsmath,amsthm}
\usepackage{graphicx,mathtools,xcolor}
\usepackage[hidelinks]{hyperref}
\usepackage{tikz, soul}
\usepackage{subcaption}
\usetikzlibrary{matrix, fit, positioning, calc}
\usepackage{cleveref}

\definecolor{newgreen}{RGB}{59, 177, 67}
\definecolor{ng}{RGB}{59, 177, 67}
\definecolor{b}{RGB}{0,122,255}

\newtheorem{theorem}{Theorem}[section]
\newtheorem{lemma}[theorem]{Lemma}
\newtheorem{corollary}[theorem]{Corollary}

\newtheorem{problem}[theorem]{Problem}

\theoremstyle{definition}
\newtheorem{definition}[theorem]{Definition}
\newtheorem{example}[theorem]{Example}

\newtheorem{manualtheoreminner}{Theorem}
\newenvironment{manualtheorem}[1]{
  \IfBlankTF{#1}
    {\renewcommand{\themanualtheoreminner}{\unskip}}
    {\renewcommand\themanualtheoreminner{#1}}
  \manualtheoreminner
}{\endmanualtheoreminner}

\newcommand{\X}{\mathcal{X}}
\newcommand{\W}{\mathcal{W}}
\newcommand{\D}{\mathcal{D}}
\newcommand{\R}{\mathcal{R}}
\newcommand{\Ss}{\mathcal{S}}
\newcommand{\RC}{\mathrm{RC}}
\newcommand{\on}{\operatorname}

\newcommand\SmallArray[1]{{
  \tiny \arraycolsep=0.08\arraycolsep\ensuremath{\begin{array}{c}#1\end{array}}}}

\title{Generalising Latin square orthogonality and Frobenius-K\"onig with alternating sign matrices }

\author{Alena Ernst, Stefano Lia, Cian O'Brien, John Sheekey,\\ and Jens Zumbrägel}

\date{}

\begin{document}

\maketitle

\renewcommand{\thefootnote}{} 

\renewcommand{\thefootnote}{\arabic{footnote}} 

\begin{abstract}
The theory of Latin squares has a long history. While the objects themselves appeared earlier, the study of their general mathematical theory dates back to Euler in the 18th century. Latin squares can be interpreted as 3-dimensional permutation hypermatrices, and alternating sign matrices often arise as a natural generalisation of permutation matrices. In 2018, Brualdi and Dahl introduced a generalisation of classical Latin squares using alternating sign hypermatrices.
Inspired by their definition, we develop the theory of Italian squares, a related  generalisation of Latin squares, together with a notion of orthogonality that resolves an inconsistency in the definition of Brualdi and Dahl. Building on classical questions from Latin square theory, we obtain results including upper bounds on the maximal size of a pairwise orthogonal set, conditions for the existence of an orthogonal mate, infinite families of orthogonal pairs, and transversals. As part of our exploration of alternating sign matrices, we also prove a Frobenius–K\"onig type result for a class of $(0, \pm1)$-matrices.
\end{abstract}

\section{Introduction}

A Latin square of order~$n$ is an $n \times n$ matrix containing~$n$ distinct symbols such that each row and column contains each symbol exactly once.  Two Latin squares of order~$n$ are orthogonal if, when the squares are superimposed, each of the $n^2$ ordered pairs of symbols appears exactly once.
Mutually orthogonal Latin squares (MOLS) have been studied for centuries, and they have interesting connections to combinatorics, finite geometry and applied algebra, see for example~\cite{KeedwellDenes}.
Despite their elementary nature, there are many associated open problems, the most famous being the prime power conjecture.  It states that a set of $n \!-\! 1$ MOLS of order $n$ (which is easily seen to be maximal) exists if and only if~$n$ is a prime power.

Alternating sign matrices (ASMs) have been the subject of sustained research interest in recent decades, arising in a number of different contexts \cite{MillsRobbinsRumsey83, LascouxSchutzenberger96} as a natural generalisation of permutation matrices. One may view  a Latin square of order~$n$ as an $n \times n \times n$ permutation hypermatrix, so that each row, column, and vertical line contains exactly a single~$1$ and otherwise~$0$.  Moreover, the orthogonality condition of Latin squares means that for any two matrix layers corresponding to different hypermatrices, their inner product equals one.

This interpretation leads to generalisations of orthogonal Latin squares in terms of alternating sign matrices.  Accordingly, in an $n \times n \times n$ alternating sign hypermatrix the non-zero entries $\pm 1$ of each row, column, and vertical line alternate in sign, starting and ending with $+1$.
These matrices were introduced by Brualdi and Dahl~\cite{BrualdiDahl18}, and a compressed version of these are \emph{Latin-like squares}, which they study.
Here, we consider alternating sign hypermatrices in a form that more closely resembles a Latin square, which we call an \emph{Italian square} (as they can be seen as a “modern version” of the Latin ones).  Our notion seems more appropriate in the context of orthogonality, which can be defined by the matrix inner product the same way as Latin squares.
 
In this work, we address and resolve several open problems from Brualdi and Dahl~\cite{BrualdiDahl18, BrualdiDahl24}, relating to the study of pairwise orthogonal Italian squares (POIS).
Along the way we study \emph{transversals}, which are $n \times n$ alternating sign matrices that are orthogonal to each matrix layer of a given Italian square.
The Frobenius-K\"onig theorem, which states when a $(0, 1)$-matrix contains a permutation matrix, plays a role in our characterisation of orthogonal alternating sign matrices, and we also state a variation of this theorem for ASMs. 

Here is an outline of the paper.  In Section~\ref{sec:background} we give some background and define Italian squares and orthogonality, comparing it to Brualdi and Dahl's notion of Latin-like squares.
Subsequently, in Section~\ref{sec:mate} we give a characterisation of those alternating sign matrices that possess an orthogonal alternating sign matrix.

\begin{manualtheorem}{\ref{thm: ortho mate}}

An ASM has an orthogonal mate if and only if it is not one of the following exceptions. 
\[ \begin{pmatrix} 1 & 0 \\ 0 & 1 \end{pmatrix} , \quad
\begin{pmatrix} 0 & 1 \\ 1 & 0 \end{pmatrix} , \quad
\begin{pmatrix} 0 & 1 & 0 \\ 1 & \!-1\! & 1 \\ 0 & 1 & 0 \end{pmatrix} \]
Moreover, every ASM with a mate is orthogonal to a permutation matrix.
\end{manualtheorem}

In Section~\ref{sec:frob}, we state variations of the Frobenius-K\"onig theorem for $(0, \pm 1)$-matrices with row/column sums in $\{0, \pm 1\}$ and alternating non-zero entries in each row/column. In particular, we give a Frobenius-K\"onig type condition for when it is possible to replace some zero entries in a given $(0, 1)$-matrix (resp.\ $(0, -1)$-matrix) with~$-1$ (resp.~$+1$) such that the resulting matrix is an ASM.

Moving on to $n \times n \times n$ hypermatrices, we study in Section~\ref{sec:transversals} transversals and examine those of the \emph{diamond} alternating sign hypermatrix in particular.

\begin{manualtheorem}{\ref{thm: diamond no transversal}}
    The diamond Italian square $\mathcal D_n$ possesses a transversal if and only if $n \equiv \pm 1 \bmod 6$.
\end{manualtheorem}

Then we turn to pairwise orthogonal Italian squares and prove in Section~\ref{sec:pois} that the general upper bound on the number of POIS coincides with that of MOLS.

\begin{manualtheorem}{\ref{thm:bound_on_POIS_size}}
    There are at most $n \!-\! 1$ POIS of order~$n$.
\end{manualtheorem}

In Section 7, we extend a well-known Kronecker-like product of Latin squares to Italian squares, and use this to show that for infinitely many orders~$n$, there exist pairs of orthogonal Italian squares of order~$n$ that are not Latin squares.

\begin{manualtheorem}{\ref{thm:explicit}}
    There exists a pair of POIS of order $k m$, neither of which is a Latin square, for $k \ne 2$ and $m \in \{ 5, 6, 7, 11 \}$. Moreover, there exists a set of three POIS of order~$7 k$, none of which is a Latin square, for all $k \not\in \{ 2, 3, 6, 10 \}$.
\end{manualtheorem}

Finally, we list some open problems and future work in Section~\ref{sec:conclude}.

\section{Background}\label{sec:background}

In this section, we provide some preliminaries and background on Latin squares and orthogonality, as well as alternating sign matrices and Italian squares.

\subsection{Latin squares and orthogonality}

For a positive integer $n$, a \emph{Latin square} of order $n$ is an $n \times n$ array containing $n$ different symbols, each occurring exactly once in each column and each row. Typically, symbols are taken to be $[n] := \{1,2,\dots,n\}$.

\begin{definition}
Two Latin squares $A$ and $B$ of the same order are \emph{orthogonal} if, for every symbol $a$ in $A$ and every symbol $b$ in $B$, there is exactly one position where $A$ has the entry $a$ and $B$ has the entry $b$. A set of Latin squares in which each distinct pair is orthogonal is called a set of \emph{mutually orthogonal Latin squares (MOLS)}. 
\end{definition}

MOLS are often alternatively referred to as \emph{pairwise} orthogonal Latin squares.  See Example~\ref{construction n=5} for an example of a set of four MOLS of order~5.

Orthogonality is one of the most important concepts within the study of Latin squares, with constructions for sets of MOLS and bounds on their cardinality being of particular historical interest. The study of MOLS was first made popular by the \emph{36 Officers Problem}, a puzzle in which officers of 6 different ranks and regiments are to be arranged in a $6 \times 6$ configuration so that each row and column contains an officer of each regiment and each rank. This is the problem that Euler first considered, and he concluded it was impossible (without proof). 

There exist only two Latin squares of order 2, which are not orthogonal to one another. Euler conjectured that there is no pair of MOLS of order~$n$ for $n \equiv 2 \bmod 4$, and proved that there exists a pair for all other $n \ge 2$. In 1901, it was proven that no pair exist for $n=6$ \cite{Tarry1901}. In 1959, Euler's conjecture was proven false for $n>6$, and consequently a pair of MOLS exists for all $n \not\in\{1,2,6\}$ \cite{BoseShrikandeParker60}.

Given a set $S = \{ L_1, L_2, \dots, L_m \}$ of~$m$ MOLS of order~$n$, the following simple argument \cite[Th.~22.1]{VanLintWilson92} proves that $m \le n \!-\! 1$. Without loss of generality, the symbols in each $L_i$ may be swapped so that the first row is $\begin{array}{|c|c|c|c|}\hline 1 & 2 & \dots & n \\\hline \end{array}$. Since $(L_i)_{1,k} = (L_j)_{1,k}$ for all $k\in[n]$, no pair in $S$ may contain the same symbol in any position outside of row 1. Since the $(1,n)$-entry is $n$, each $(2,n)$-entry must be a distinct element of $[n \!-\! 1]$, and therefore $m \le n \!-\! 1$.

The finite field based construction, cf.~\cite[Sec.~22]{VanLintWilson92}, achieves this bound for prime power $n$. Indeed, for $k \in [n \!-\! 1]$, define $(L_k)_{i,j} = ki+j$ for each $i, j \in \mathbb{F}_n$. Then $\{L_1, L_2, \dots, L_{n-1} \}$ is a set of $n \!-\! 1$ MOLS of order~$n$.

\begin{example}\label{construction n=5} 
The previous finite field based construction for $n=5$ generates the following orthogonal Latin squares.
\[ \begin{array}{|c|c|c|c|c|}\hline 
0 & 1 & 2 & 3 & 4 \\\hline 
1 & 2 & 3 & 4 & 0 \\\hline 
2 & 3 & 4 & 0 & 1 \\\hline 
3 & 4 & 0 & 1 & 2 \\\hline 
4 & 0 & 1 & 2 & 3 \\\hline 
\end{array} \quad
\begin{array}{|c|c|c|c|c|}\hline
0 & 1 & 2 & 3 & 4 \\\hline
2 & 3 & 4 & 0 & 1 \\\hline
4 & 0 & 1 & 2 & 3 \\\hline
1 & 2 & 3 & 4 & 0 \\\hline
3 & 4 & 0 & 1 & 2 \\\hline
\end{array} \quad
\begin{array}{|c|c|c|c|c|}\hline
0 & 1 & 2 & 3 & 4 \\\hline
3 & 4 & 0 & 1 & 2 \\\hline
1 & 2 & 3 & 4 & 0 \\\hline
4 & 0 & 1 & 2 & 3 \\\hline
2 & 3 & 4 & 0 & 1 \\\hline
\end{array} \quad
\begin{array}{|c|c|c|c|c|}\hline
0 & 1 & 2 & 3 & 4 \\\hline
4 & 0 & 1 & 2 & 3 \\\hline
3 & 4 & 0 & 1 & 2 \\\hline
2 & 3 & 4 & 0 & 1 \\\hline
1 & 2 & 3 & 4 & 0 \\\hline
\end{array} \]
\end{example}

\subsection{Alternating sign matrices and Latin-like squares}
\label{subsec:alternsignmat}

\begin{definition}
An \emph{alternating sign matrix (ASM)} is a $(0, 1, -1)$-matrix with the property that the non-zero entries in each row and column alternate in sign, beginning and ending with $+1$.
\end{definition}

An immediate consequence of this definition is that each row and column of an ASM sums to 1, and therefore each ASM is a square matrix. Permutation matrices are examples of ASMs, and ASMs arise in a number of contexts as a natural generalisation of permutations. For example, ASMs of order $n$ are the minimal lattice extension (Dedekind-MacNeille completion) of the symmetric group $S_n$ of permutations on $n$ symbols under the \emph{Bruhat order} \cite{LascouxSchutzenberger96}.

While permutation matrices are the ASMs with the least number of non-zero entries, the \emph{diamond ASMs} have the greatest number. There is a unique \emph{diamond ASM} $D_n$ of odd order $n$, and a pair of diamond ASMs $D_n$ and $D_n^\prime$ for even $n$ which are the reflections of each other. $D_n$ is defined as follows. For $i,j \in [n]$, the $(i,j)$-entry of $D_n$ is defined as follows.
\[ (D_n)_{ij} = \begin{cases}
    (-1)^{i + j + \lceil \frac n 2 \rceil + 1} & \text{if } |i \!+\! j \!-\! n \!-\! 1| \le n \!-\! \lceil \frac n 2 \rceil \text{ and } |i \!-\! j| \le \lceil \frac n 2 \rceil \!-\! 1 \\
    0 & \text{otherwise}
\end{cases} \]

For example,
\[
D_5 = \begin{pmatrix}
       &   & + &   &   \\
       & + & - & + &   \\
     + & - & + & - & + \\
       & + & - & + &   \\
       &   & + &   & 
\end{pmatrix} , \quad
D_6 = \begin{pmatrix}
      &   & + &   &   &   \\
      & + & - & + &   &   \\
    + & - & + & - & + &   \\
      & + & - & + & - & + \\
      &   & + & - & + &   \\
      &   &   & + &   & 
\end{pmatrix} ,
\]
where here and henceforth we abbreviate $+1$ and $-1$ entries by $+$ and $-$, and leave 0 entries blank.

Moreover, the number of entries in row/column $1, 2, 3, \dots, n \!-\! 2, n \!-\! 1, n$ of an $n \times n$ ASM is at most $1, 3, 5, \dots, 5, 3, 1$.  The diamond ASMs are the unique ASMs achieving this bound.

A \emph{hypermatrix} of order $n$ is an $n \times n \times n$ array $H = (h_{ijk})$ with rows (fixing $i,k$ and varying $j$), columns (varying $i$), and \emph{vertical lines} (varying $k$). The $k$th \emph{plane} $H_k$ of $H$ is the matrix for fixed $k$ and varying $i,j$. We also write $H = (H_1, H_2, \dots, H_n)$.

A \emph{permutation hypermatrix} $P$ of order $n$ has entries from $\{0,1\}$ and has only one non-zero entry in each row, each column, and each vertical line. The set of Latin squares of order $n$ with symbols $[n]$ is in the following simple bijection with the set of 3-dimensional \emph{permutation hypermatrices}. The positions in the Latin square $L$ of the symbol $k\in[n]$ correspond to the positions of the $+1$ entries in plane $P_k$ of the permutation hypermatrix $P$.
\[ L(P) = \sum_{k=1}^n k P_k \]

\begin{example} Latin square $L$ corresponding to permutation hypermatrix $P$.
\[ P = \left( \begin{pmatrix}
    + &  &   \\
      &  & + \\
      & + & 
\end{pmatrix},
\begin{pmatrix}
      & + &   \\
    + &   &   \\
      &   & +
\end{pmatrix},
\begin{pmatrix}
      &   & + \\
      & + &   \\
    + &   & 
\end{pmatrix} \right) \iff L = \begin{array}{|c|c|c|}\hline
1 & 2 & 3 \\\hline
2 & 3 & 1 \\\hline
3 & 1 & 2 \\\hline
\end{array} \]
\[ \sum_{k=1}^n k P_k =
\begin{pmatrix}
      1 & 0 & 0 \\
      0 & 0 & 1 \\
      0 & 1 & 0 \end{pmatrix}
+ 2 \begin{pmatrix}
      0 & 1 & 0 \\
      1 & 0 & 0 \\
      0 & 0 & 1 \end{pmatrix}
+ 3 \begin{pmatrix}
      0 & 0 & 1 \\
      0 & 1 & 0 \\
      1 & 0 & 0 \end{pmatrix} =
\begin{pmatrix}
1 & 2 & 3 \\
2 & 3 & 1 \\
3 & 1 & 2 \\
\end{pmatrix} \]
\end{example}

Motivated by this observation, Brualdi and Dahl \cite{BrualdiDahl18} defined the following 3-dimensional analogues of ASMs and used them to generalise Latin squares.

\begin{definition}
    An \emph{alternating sign hypermatrix (ASHM)} is a hypermatrix $(a_{ijk})$ with entries in $\{0,1,-1\}$ such that the non-zero entries in each row, column, and vertical line alternate in sign, beginning and ending with $+1$. Given an ASHM $A$ we write $A= (A_k)=(A_1,A_2, \ldots, A_n)$ where $A_k$ denotes the $k$-th plane of $A$.
\end{definition}

An \emph{ASHM Latin-like square} (\emph{Latin-like  square} for short) $L$ of order $n$ is then an $n \times n$ array such that, for some ASHM $A = (A_1,A_2,\dots,A_n)$,
    \[L(A) =\sum_{k=1}^n k A_k.\]

\begin{example}
    The following ASHM $A = (A_1,A_2,A_3)$ and its reflection $(A_3,A_2,A_1)$ are the only $n \times n \times n$ ASHM of order $n \le 3$ containing a negative entry.
\[ A = \left( \begin{pmatrix}
    + &   &   \\
      & + &   \\
      &   & +
\end{pmatrix},
\begin{pmatrix}
      & + &   \\
    + & - & + \\
      & + & 
\end{pmatrix},
\begin{pmatrix}
      &   & + \\
      & + &   \\
    + &   & 
\end{pmatrix} \right) \]
\[ L(A) = 
\begin{pmatrix}
    1 & 0 & 0 \\
    0 & 1 & 0 \\
    0 & 0 & 1
\end{pmatrix}
+ 2 \begin{pmatrix}
    0 & 1 & 0 \\
    1 & \!-1\! & 1 \\
    0 & 1 & 0
\end{pmatrix}
+ 3 \begin{pmatrix}
    0 & 0 & 1 \\
    0 & 1 & 0 \\
    1 & 0 & 0
\end{pmatrix} =
\begin{pmatrix}
    1 & 2 & 3 \\
    2 & 2 & 2 \\
    3 & 2 & 1
\end{pmatrix} \]
\end{example}

Let $\langle A, B \rangle$ denote the Frobenius inner product of the $n \times n$ matrices $A$ and $B$, i.e.\ $\langle A, B \rangle = \text{tr}(AB^\top)$. Brualdi and Dahl noted in \cite{BrualdiDahl18} that given permutation hypermatrices $P$ and $Q$ of order $n$ with planes $(P_k)$ and $(Q_k)$, respectively, Latin squares $L(P)$ and $L(Q)$ are orthogonal if and only if $\langle P_i, Q_j \rangle = 1$ for all $i,j \in[n]$. For ASHMs $A=(A_k)$ and $B=(B_k)$, they proposed to define Latin-like squares $L(A)$ and $L(B)$ to be orthogonal if $\langle A_i, B_j \rangle = 1$ for all $i,j \in[n]$.

They note the existence of the pair of orthogonal Latin-like squares of order $6$ as given in Figure~\ref{fig:orthog_6} (with respective corresponding alternating sign hypermatrices), and since there exists no pair of MOLS of order 6, they suggested that there may exist other sets of mutually orthogonal Latin-like squares of order $n$ whose cardinality is greater than the best possible for MOLS of order $n$.

\begin{figure}[h]
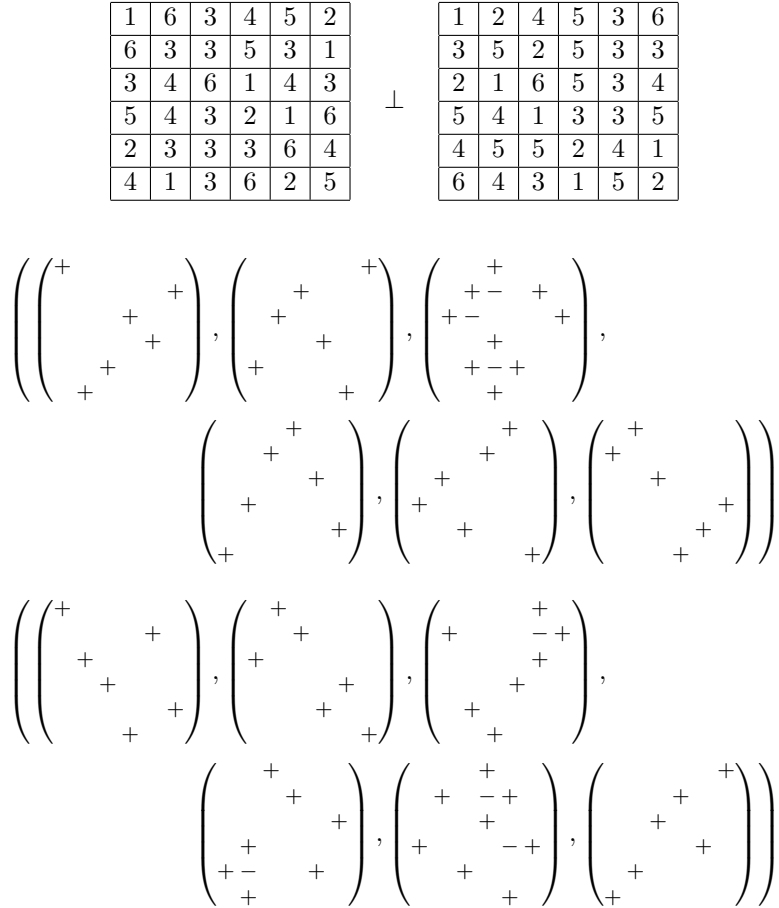

\[ \begin{array}{|c|c|c|c|c|c|}\hline
1 & 6 & 3 & 4 & 5 & 2 \\\hline
6 & 3 & 3 & 5 & 3 & 1 \\\hline
3 & 4 & 6 & 1 & 4 & 3 \\\hline
5 & 4 & 3 & 2 & 1 & 6 \\\hline
2 & 3 & 3 & 3 & 6 & 4 \\\hline
4 & 1 & 3 & 6 & 2 & 5 \\\hline
\end{array} \quad\perp\quad
\begin{array}{|c|c|c|c|c|c|}\hline
1 & 2 & 4 & 5 & 3 & 6 \\\hline
3 & 5 & 2 & 5 & 3 & 3 \\\hline
2 & 1 & 6 & 5 & 3 & 4 \\\hline
5 & 4 & 1 & 3 & 3 & 5 \\\hline
4 & 5 & 5 & 2 & 4 & 1 \\\hline
6 & 4 & 3 & 1 & 5 & 2 \\\hline
\end{array} \]
\begin{multline*} {\footnotesize \qquad \left( \begin{pmatrix}
 +  ~~~ ~~~ ~~~ ~~~ ~~~ \\
~~~ ~~~ ~~~ ~~~ ~~~  +  \\
~~~ ~~~ ~~~  +  ~~~ ~~~ \\
~~~ ~~~ ~~~ ~~~  +  ~~~ \\
~~~ ~~~  +  ~~~ ~~~ ~~~ \\
~~~  +  ~~~ ~~~ ~~~ ~~~
\end{pmatrix},
\begin{pmatrix}
~~~ ~~~ ~~~ ~~~ ~~~  +  \\
~~~ ~~~  +  ~~~ ~~~ ~~~ \\
~~~  +  ~~~ ~~~ ~~~ ~~~ \\
~~~ ~~~ ~~~  +  ~~~ ~~~ \\
 +  ~~~ ~~~ ~~~ ~~~ ~~~ \\
~~~ ~~~ ~~~ ~~~  +  ~~~
\end{pmatrix},
\begin{pmatrix}
~~~ ~~~  +  ~~~ ~~~ ~~~ \\
~~~  +   -  ~~~  +  ~~~ \\
 +   -  ~~~ ~~~ ~~~  +  \\
~~~ ~~~  +  ~~~ ~~~ ~~~ \\
~~~  +   -   +  ~~~ ~~~ \\
~~~ ~~~  +  ~~~ ~~~ ~~~
\end{pmatrix}, \right. } \\
{\footnotesize \left. \begin{pmatrix}
~~~ ~~~ ~~~  +  ~~~ ~~~ \\
~~~ ~~~  +  ~~~ ~~~ ~~~ \\
~~~ ~~~ ~~~ ~~~  +  ~~~ \\
~~~  +  ~~~ ~~~ ~~~ ~~~ \\
~~~ ~~~ ~~~ ~~~ ~~~  +  \\
 +  ~~~ ~~~ ~~~ ~~~ ~~~ 
\end{pmatrix} ,
\begin{pmatrix}
~~~ ~~~ ~~~ ~~~  +  ~~~ \\
~~~ ~~~ ~~~  +  ~~~ ~~~ \\
~~~  +  ~~~ ~~~ ~~~ ~~~ \\
 +  ~~~ ~~~ ~~~ ~~~ ~~~ \\
~~~ ~~~  +  ~~~ ~~~ ~~~ \\
~~~ ~~~ ~~~ ~~~ ~~~  +  
\end{pmatrix} ,
\begin{pmatrix}
~~~  +  ~~~ ~~~ ~~~ ~~~ \\
 +  ~~~ ~~~ ~~~ ~~~ ~~~ \\
~~~ ~~~  +  ~~~ ~~~ ~~~ \\
~~~ ~~~ ~~~ ~~~ ~~~  +  \\
~~~ ~~~ ~~~ ~~~  +  ~~~ \\
~~~ ~~~ ~~~  +  ~~~ ~~~
\end{pmatrix} \right) \qquad} \end{multline*}
\vspace{-5mm}
\begin{multline*} {\footnotesize \qquad \left( \begin{pmatrix}
 +  ~~~ ~~~ ~~~ ~~~ ~~~ \\
~~~ ~~~ ~~~ ~~~  +  ~~~ \\
~~~  +  ~~~ ~~~ ~~~ ~~~ \\
~~~ ~~~  +  ~~~ ~~~ ~~~ \\
~~~ ~~~ ~~~ ~~~ ~~~  +  \\
~~~ ~~~ ~~~  +  ~~~ ~~~ 
\end{pmatrix},
\begin{pmatrix}
~~~  +  ~~~ ~~~ ~~~ ~~~ \\
~~~ ~~~  +  ~~~ ~~~ ~~~ \\
 +  ~~~ ~~~ ~~~ ~~~ ~~~ \\
~~~ ~~~ ~~~ ~~~  +  ~~~ \\
~~~ ~~~ ~~~  +  ~~~ ~~~ \\
~~~ ~~~ ~~~ ~~~ ~~~  +  
\end{pmatrix} ,
\begin{pmatrix}
~~~ ~~~ ~~~ ~~~  +  ~~~ \\
 +  ~~~ ~~~ ~~~  -   +  \\
~~~ ~~~ ~~~ ~~~  +  ~~~ \\
~~~ ~~~ ~~~  +  ~~~ ~~~ \\
~~~  +  ~~~ ~~~ ~~~ ~~~ \\
~~~ ~~~  +  ~~~ ~~~ ~~~ 
\end{pmatrix} , \right.} \\
{\footnotesize \left.\begin{pmatrix}
~~~ ~~~  +  ~~~ ~~~ ~~~ \\
~~~ ~~~ ~~~  +  ~~~ ~~~ \\
~~~ ~~~ ~~~ ~~~ ~~~  +  \\
~~~  +  ~~~ ~~~ ~~~ ~~~ \\
 +   -  ~~~ ~~~  +  ~~~ \\
~~~  +  ~~~ ~~~ ~~~ ~~~ 
\end{pmatrix},
\begin{pmatrix}
~~~ ~~~ ~~~  +  ~~~ ~~~ \\
~~~  +  ~~~  -   +  ~~~ \\
~~~ ~~~ ~~~  +  ~~~ ~~~ \\
 +  ~~~ ~~~ ~~~  -   +  \\
~~~ ~~~  +  ~~~ ~~~ ~~~ \\
~~~ ~~~ ~~~ ~~~  +  ~~~ 
\end{pmatrix},
\begin{pmatrix}
~~~ ~~~ ~~~ ~~~ ~~~  +  \\
~~~ ~~~ ~~~  +  ~~~ ~~~ \\
~~~ ~~~  +  ~~~ ~~~ ~~~ \\
~~~ ~~~ ~~~ ~~~  +  ~~~ \\
~~~  +  ~~~ ~~~ ~~~ ~~~ \\
 +  ~~~ ~~~ ~~~ ~~~ ~~~ 
\end{pmatrix} \right) \qquad} \end{multline*}
\caption{Pair of orthogonal Latin-like squares of order~$6$ with their corresponding alternating sign hypermatrices, from~\cite{BrualdiDahl18}.}
\label{fig:orthog_6}
\end{figure}

\subsection{Italian squares and orthogonality}

Just as each Latin square is in simple bijection with a permutation hypermatrix, we propose the following definition of an \emph{Italian square}, adapted from \cite{Obrien24}, which correspond to an ASHM by the same simple bijection.

\begin{definition}
An \emph{Italian square} of order $n$ is an $n \times n$ array, with each cell containing a signed formal sum of a subset of symbols from $[n]$, such that each formal sum in ascending order, and each symbol in a row or column, alternates in sign beginning and ending with $+$. 
\end{definition}

We denote the ASM corresponding to symbol $k \in [n]$ in the Italian square $L$ by $L_k$ and write $L=(L_1,L_2,\ldots, L_n)$.

\begin{example} An ASHM $(L_1, L_2, L_3)$ and its corresponding Italian square $L$.
\[ \left( \begin{pmatrix}
    + &   &   \\
      & + &   \\
      &   & +
\end{pmatrix},
\begin{pmatrix}
      & + &   \\
    + & - & + \\
      & + & 
\end{pmatrix},
\begin{pmatrix}
      &   & + \\
      & + &   \\
    + &   & 
\end{pmatrix} \right) \iff \begin{array}{|c|c|c|}\hline
    1 & 2 & 3 \\\hline
    2 & \SmallArray{1-2\\+3} & 2 \\\hline
    3 & 2 & 1 \\\hline
\end{array} \]
\end{example}

This is motivated by the observation that Brualdi and Dahl's definition of orthogonality is not consistent for Latin-like squares, but is consistent when instead applied to Italian squares. This arises due to the fact that different ASHMs can have the same Latin-like square, as demonstrated in \cite{Obrien20}.

\begin{example}\label{inconsistency}
    The following are three Italian squares, respectively denoted $A$, $B$, and $C$ (where $C$ is also a Latin square).
\[ {\small \begin{array}{|c|c|c|c|c|}\hline
    1 & 2 & 3 & 4 & 5 \\\hline
    2 & \SmallArray{1-2\\+3} & \SmallArray{2-3\\+5} & 3 & 4 \\\hline
    3 & \SmallArray{2-3\\+4} & \SmallArray{1-2\\+3} & \SmallArray{2-3\\+5} & 3 \\\hline
    5 & 3 & \SmallArray{2-3\\+4} & \SmallArray{1-2\\+3} & 2 \\\hline
    4 & 5 & 3 & 2 & 1 \\\hline
\end{array} \qquad
\begin{array}{|c|c|c|c|c|}\hline
    1 & 2 & 3 & 4 & 5 \\\hline
    2 & \SmallArray{1-2\\+3} & 4 & \SmallArray{2-4\\+5} & 4 \\\hline
    3 & \SmallArray{2-3\\+4} & \SmallArray{1-4\\+5} & 4 & 3 \\\hline
    5 & 3 & \SmallArray{2-3\\+4} & \SmallArray{1-2\\+3} & 2 \\\hline
    4 & 5 & 3 & 2 & 1 \\\hline
\end{array} \qquad
\begin{array}{|c|c|c|c|c|}\hline
    1 & 2 & 3 & 4 & 5 \\\hline
    4 & 5 & 2 & 3 & 1 \\\hline
    5 & 3 & 4 & 1 & 2 \\\hline
    3 & 1 & 5 & 2 & 4 \\\hline
    2 & 4 & 1 & 5 & 3 \\\hline
\end{array}} \]
    By Brualdi and Dahl's definition of orthogonality, the Latin-like square $L(A)$ is orthogonal to $C$, while $L(B)$ is not. But
    \[ L(A) = L(B)=
\begin{array}{|c|c|c|c|c|}\hline
    1 & 2 & 3 & 4 & 5 \\\hline
    2 & 2 & 4 & 3 & 4 \\\hline
    3 & 3 & 2 & 4 & 3 \\\hline
    5 & 3 & 3 & 2 & 2 \\\hline
    4 & 5 & 3 & 2 & 1 \\\hline
\end{array} ~. \]
\end{example}

We therefore propose the following alternative definition of orthogonality.

\begin{definition}
    Italian squares $A$ and $B$ of order $n$ are orthogonal if $\langle A_i, B_j \rangle = 1$ for all $i,j \in[n]$,where $A_i$ and $B_j$ are the planes of $A$ and $B$, respectively. A set of \emph{pairwise orthogonal Italian squares (POIS)} is defined as expected.
\end{definition}

In the Example \ref{inconsistency}, we would then say $A$ is orthogonal to $C$, while $B$ is not.

\begin{example}\label{ex:noextend}
    The following Latin square is orthogonal to both of the non-Latin Italian squares (which are not orthogonal to each other).
\[ \begin{array}{|c|c|c|c|}\hline
    1 & 3 & 2 & 4 \\\hline
    3 & \SmallArray{2-3\\+4} & \SmallArray{1-2\\+3} & 2 \\\hline
    2 & \SmallArray{1-2\\+3} & \SmallArray{2-3\\+4} & 3 \\\hline
    4 & 2 & 3 & 1 \\\hline
\end{array} ~\perp~
\begin{array}{|c|c|c|c|}\hline
    1 & 2 & 3 & 4 \\\hline
    2 & 1 & 4 & 3 \\\hline
    4 & 3 & 2 & 1 \\\hline
    3 & 4 & 1 & 2 \\\hline
\end{array} ~\perp~
\begin{array}{|c|c|c|c|}\hline
    1 & 4 & 2 & 3 \\\hline
    3 & 2 & \SmallArray{1-2\\+4} & 2 \\\hline
    2 & \SmallArray{1-2\\+3} & 2 & 4 \\\hline
    4 & 2 & 3 & 1 \\\hline
\end{array} \]

    By exhaustive computer search, these are the unique examples (up to equivalence under symmetry of the unit cube) of a pair of order $4$ which cannot be extended to a triple of POIS.
\end{example}

\begin{example}\label{6x6}
The following are the Italian squares corresponding to Brualdi and Dahl's $6 \times 6$ construction, see Figure~\ref{fig:orthog_6}.
\[ \begin{array}{|c|c|c|c|c|c|}\hline
1 & 6 & 3 & 4 & 5 & 2 \\\hline
6 & 3 & \SmallArray{2-3\\+4} & 5 & 3 & 1 \\\hline
3 & \SmallArray{2-3\\+5} & 6 & 1 & 4 & 3 \\\hline
5 & 4 & 3 & 2 & 1 & 6 \\\hline
2 & 3 & \SmallArray{1-3\\+5} & 3 & 6 & 4 \\\hline
4 & 1 & 3 & 6 & 2 & 5 \\\hline
\end{array} \quad\perp\quad
\begin{array}{|c|c|c|c|c|c|}\hline
1 & 2 & 4 & 5 & 3 & 6 \\\hline
3 & 5 & 2 & \SmallArray{4-5\\+6} & \SmallArray{1-3\\+5} & 3 \\\hline
2 & 1 & 6 & 5 & 3 & 4 \\\hline
5 & 4 & 1 & 3 & \SmallArray{2-5\\+6} & 5 \\\hline
4 & \SmallArray{3-4\\+6} & 5 & 2 & 4 & 1 \\\hline
6 & 4 & 3 & 1 & 5 & 2 \\\hline
\end{array} \]
\end{example}

\section{ASMs with an orthogonal mate}\label{sec:mate}

If ASMs~$A$ and~$B$ satisfy $\langle A, B \rangle = 1$, we say they are each other's \emph{orthogonal mate}. An example of an ASM without an orthogonal mate is the diamond ASM
\[ D_3 = \begin{pmatrix} 0 & 1 & 0 \\ 1 & \!-1\! & 1 \\ 0 & 1& 0 \end{pmatrix} . \]
This can be easily verified, as all $3\times3$ permutation matrices $A$ with a $+1$ in the $(2,2)$ position satisfy $\langle A, D_3 \rangle = -1$, and all without satisfy $\langle A, D_3 \rangle = 2$. The following are the only $2 \times 2$ ASMs. Neither are their own nor each other's orthogonal mate.
\[ I_2 = \begin{pmatrix} 1 & 0 \\ 0 & 1 \end{pmatrix} , \quad I^\prime_2 = \begin{pmatrix} 0 & 1 \\ 1 & 0 \end{pmatrix} \]

Brualdi and Dahl posed the following problems \cite[Prob.~46(a)]{BrualdiDahl18}.
\begin{enumerate}
\item[] Which $n \times n$ ASMs have an orthogonal ASM-mate? Which $n \times n$ ASMs have a permutation matrix as an orthogonal mate?
\end{enumerate}

The following result answers both of these questions. For the proof, we recall the definition of a \emph{T-block} \cite{BrualdiSchroeder17}, which is an $n \times n$ matrix $T_{i_1,j_1;i_2,j_2}$ for which the only non-zero entries are $+1$ in the $(i_1,j_1)$ and $(i_2,j_2)$ positions, and $-1$ in the $(i_1,j_2)$ and $(i_2,j_1)$ positions, $1 \le i_1,j_1,i_2,j_2 \le n$.

We also recall the Frobenius-K\"onig Theorem, which characterises when a given $n \times n$ $(0,1)$-matrix contains a permutation matrix (in the sense that there is a set of $n$ entries equal to 1 which occupy distinct rows and columns). Equivalent results were proven independently by Frobenius, K\"onig, and Hall \cite{Hall35}. We refer to \cite{Schneider77} for a more detailed historical account.

\begin{theorem}\label{thm: ortho mate}
An ASM has an orthogonal mate if and only if it is not $I_2$, $I^\prime_2$, or $D_3$. Moreover, every ASM with an orthogonal mate is orthogonal to a permutation matrix.
\end{theorem}

\begin{proof}
There is a single $1 \times 1$ ASM, which is self-orthogonal. There are two $2 \times 2$ ASMs, $I_2$ and $I^\prime_2$, neither of which has an orthogonal mate. Each $3 \times 3$ permutation matrix $P$ has an orthogonal mate; any permutation resulting from swappping two rows of $P$. The only $3 \times 3$ ASM which is not a permutation matrix is $D_3$, which has no orthogonal mate.

Suppose $A$ has order $n\ge4$ and contains an entire row or column of non-zero entries. Then $n = 2m-1$ for some $m \in \mathbb{N}$, and we construct a permutation matrix $P$ for which $P_{mm} = 1$, and $P_{ij} = 0$ for all other $A_{ij}\ne0$. Without loss of generality, suppose $A$ contains an entire row of non-zero entries. The Frobenius-K\"onig Theorem implies that such a $P$ exists if and only if the matrix resulting from replacing the central entry of $A$ with 0 contains no $r \times s$ $(1,-1)$-submatrix for some $r+s=n+1$. Any set of $r$ rows contains a row at least $\lfloor\frac{r}{2}\rfloor$ from the central row. This row contains at least $r$ zeros, and therefore at most $n-r$ non-zeros. So any $r \times s$ $(1,-1)$-submatrix satisfies $r+s \le r+(n-r) = n$, and therefore such a $P$ exists.

If $n \equiv 1 \bmod 4$, then $A_{mm} = 1$, and $P$ is orthogonal to $A$. If $n \equiv 3 \bmod 4$, then $A_{mm} = -1$. In this case, there is some $i_1$ for which $P_{i_1,1} = 1$ and some $i_n$ for which $P_{i_n,n} = 1$. Let $j_1\ne m$ be the index of a $+1$ entry in row $i_1$ of $A$, and $j_n \ne m$ the index in row $i_n$. If the only non-zero in row $i_1$ of $A$ is in column $m$, replace $P$ with $P-T_{i_1,1;i,j}$, where $P_{i,j} = 1$, $i \ne m$, and there is a non-zero in row $i$ of $A$ which is not in column $m$. Relabel so that $i_1$ is the new row for which $P_{i_1,1} = 1$. Repeat similarly for $i_n$ if its only non-zero entry is in column $m$. Now let $i_2$ be the value for which $P_{i_2,j_1} = 1$, and $i_{n-1}$ for which $P_{i_{n-1},j_n} = 1$. The permutation $Q = P-T_{i_1,1;i_2,j_1}-T_{i_n,n;i_{n-1},j_n}$ has two $+1$ entries ($Q_{i_1,j_1}$ and $Q_{i_n,j_n}$) where $A$ has $+1$ and one $+1$ ($Q_{m,m}$) entry where $A$ has $-1$, and is therefore orthogonal to $A$.

Now suppose $A$ has order $n\ge4$ and does not contain an entire row or column of non-zero entries. First we show that there exists permutation matrix $P$ for which $A_{ij} = 0$ for all $P_{ij}=1$. Since the number of non-zero entries in row $i$ of a $A$ is bounded above by $\min(2i-1, 2n-2i+1)$ and $A$ is not an odd diamond, the number of zero entries in row $i$ of $A$ is at least $\max(1, n+1-2i, 2i-n-1)$. The Frobenius-K\"onig Theorem implies that such a $P$ exists if and only if each $r \times s$ submatrix of $A$ with $r+s=n+1$ contains a 0. Again, since $A$ does not contain a full row of non-zeros, any set of $r$ rows contains a row with at least $r$ zeros (and therefore at most $n-r$ non-zeros). So any $r \times s$ submatrix of $A$ with $r+s = n+1$ contains a 0, and therefore such a $P$ exists.

We now show that there is some $T$-block switch that can be made to $P$ so the resulting permutation is orthogonal to $A$. Let $j_1$ be such that $P_{1,j_1} = 1$, and suppose the unique non-zero entry in row 1 of $A$ is in column $j_2$. If $(j_1,j_2) \in \{(1, 2), (n, n \!-\! 1)\}$, then since $n \ge 4$, there is some $(i,j) \ne (1,j_1)$ such that $P_{ij} = 1$ and $A_{i1} = 0$. In this case, replace $P$ with the permutation $P - T_{1,j_1;i,j}$, and note that $\langle A, P \rangle = 0$.

Let $i_2$ be the row of $P$ containing a 1 in column $j_2$. If $A_{i_2,j_1} = 0$, we are done and $P-T_{1,j_1;i_2,j_2}$ is orthogonal to $A$. Otherwise, all $j \not\in \{j_1,j_2\}$ satisfy $A_{1j} = P_{1j} = 0$, and if any $i\not\in\{1,i_2\}$ satisfies $A_{i,j_1} = 1$, we are done and $P-T_{1,j_1;i,j}$ is orthogonal to $A$, where $P_{ij}=1$. Otherwise $A_{i_2,j_1}$ is the only non-zero in column $j_1$, and there is some $j_3 \in \{j_1-1,j_1+1\}$ such that $j_3\ne j_2$. Let $i_3$ be such that $P_{i_3,j_3} = 1$, and consider the permutation $Q = P-T_{1,j_1;i_3,j_3}$. Since columns $j_1$ and $j_3$ are next to one another, and $A_{i_2,j_1} = 1$, it follows that $A_{i_2,j_3} \ne 1$, and therefore there is some other row $i_4$ for which $A_{i_4, j_3} = 1$. Then $Q-T_{1,j_3;i_4,j_4}$ is orthogonal to $A$, where $j_4$ is the position of the non-zero in row $i_4$ of $Q$.
\end{proof}

\section{A Frobenius-K\"onig result for $(0,\pm1)$-matrices}\label{sec:frob}

The Frobenius-K\"onig Theorem characterises when a $(0,1)$-matrix contains a permutation matrix $P$ (meaning $P \le X$ entrywise). This is possible if and only if there is no $r\times s$ zero submatrix of $X$ with $r+s = n+1$. Brualdi and Dahl \cite{BrualdiDahl24} investigate similar problems for the class $\X_n$ of $n \times n$ $(0,\pm1)$-matrices, and prove a Frobenius-K\"onig type result for such matrices with all row and column sums equal to 1. Given $k \in \{1,-1\}$ and $X \in \X_n$, they denote by $X(k)$ the matrix obtained by replacing all $-k$ entries by 0, and pose the following related questions \cite[Question 1.3]{BrualdiDahl24}.
\begin{enumerate}
    \item[(I)] Given a $(0,1)$-matrix $X \in \X_n$, when does there exist an $n \times n$ ASM $A$ such that $A(1) = X$?
    \item[(II)] Given a $(0,-1)$-matrix $X \in \X_n$, when does there exist an $n \times n$ ASM $A$ such that $A(-1) = X$?
\end{enumerate}

We now answer these questions for ASMs and later for a wider class in $\X_n$.

\begin{definition}
    Let $X\in\X_n$ be a $(0,k)$-matrix, $k \in \{1,-1\}$. We define the \emph{row/column regions} of $X$ to be the contiguous 0 entries occupying a single row/column of $X$ such that the first (resp. last) entry in the region is either the first (resp. last) entry of the row/column or has a non-zero directly before (resp. after). If two non-zero entries are next to one another, we say there is an empty row/column region between them. We define type $k$ and $-k$ regions as follows.
    \begin{itemize}
        \item The type $k$ regions are the subset of row/column regions with the entry directly before and after the region equal to $k$.
        \item The type $-k$ regions are the entire set of row/column regions.
    \end{itemize}
\end{definition}

Note that this classification into type $k$ and $-k$ regions will be more useful later when generalising our result beyond ASMs. We wish to determine when a given $(0,k)$-matrix has an ASM $A$ with $A(k) = X$, and the region type is equal to the row/column sum in A. Therefore we are only interested in type 1 regions for ASMs. These are type $k$ regions for $k=1$, and type $-k$ regions for $k=-1$.

\begin{example}\label{regions}
The type 1 row and column regions of the following $(0,1)$-matrix $X \in \X_{12}$ are indicated. If $-1$ entries can be added to these regions such that exactly one is in each row region and one in each column region, the resulting matrix is an ASM.
\[ \begin{tikzpicture}
\matrix (A) [matrix of math nodes,
             left delimiter=(,
             right delimiter=),
             ampersand replacement=\&,
             row sep=0.2em,
             column sep=0.2em,
             font=\footnotesize] {
 0\& 0\& 0\& 1\& 0\& 0\& 0\& 0\& 0\& 0\& 0\& 0\\
 0\& 0\& 1\& 0\& 0\& 0\& 1\& 0\& 0\& 0\& 0\& 0\\
 0\& 0\& 0\& 0\& 1\& 0\& 0\& 0\& 0\& 0\& 0\& 0\\
 0\& 0\& 0\& 0\& 0\& 0\& 0\& 1\& 0\& 0\& 0\& 0\\
 0\& 0\& 0\& 0\& 0\& 0\& 0\& 0\& 0\& 1\& 0\& 0\\
 1\& 0\& 0\& 0\& 0\& 1\& 0\& 0\& 0\& 0\& 0\& 1\\
 0\& 0\& 0\& 1\& 0\& 0\& 0\& 0\& 1\& 0\& 0\& 0\\
 0\& 1\& 0\& 0\& 0\& 0\& 0\& 0\& 0\& 0\& 1\& 0\\
 0\& 0\& 1\& 0\& 0\& 0\& 0\& 0\& 0\& 0\& 0\& 0\\
 0\& 0\& 0\& 0\& 0\& 0\& 0\& 1\& 0\& 0\& 0\& 0\\
 0\& 0\& 0\& 0\& 0\& 0\& 0\& 0\& 0\& 1\& 0\& 0\\
 0\& 0\& 0\& 0\& 1\& 0\& 0\& 0\& 0\& 0\& 0\& 0\\
};
\node[left=4mm of A] {$X =$};
\node[draw, thick, fit=(A-2-4)(A-2-6), inner sep=0pt] {};
\node[draw, thick, fit=(A-6-2)(A-6-5), inner sep=0pt] {};
\node[draw, thick, fit=(A-6-7)(A-6-11), inner sep=0pt] {};
\node[draw, thick, fit=(A-7-5)(A-7-8), inner sep=0pt] {};
\node[draw, thick, fit=(A-8-3)(A-8-10), inner sep=0pt] {};
\node[draw, thick, fit=(A-3-3)(A-8-3), inner sep=0pt] {};
\node[draw, thick, fit=(A-2-4)(A-6-4), inner sep=0pt] {};
\node[draw, thick, fit=(A-4-5)(A-11-5), inner sep=0pt] {};
\node[draw, thick, fit=(A-5-8)(A-9-8), inner sep=0pt] {};
\node[draw, thick, fit=(A-6-10)(A-10-10), inner sep=0pt] {};
\end{tikzpicture} \]
\end{example}

\begin{example}\label{type 1 Y}
The type 1 row regions of the following $(0,-1)$-matrix $Y \in \X_9$ are indicated below. To obtain an ASM, a single $+1$ entry must be added to each type 1 row region in such a way that there is also a single $+1$ in each type 1 column region.

\[
\begin{tikzpicture}
\matrix (A) [matrix of math nodes,
             left delimiter=(,
             right delimiter=),
             ampersand replacement=\&,
             row sep=0.2em,
             column sep=0.2em,
             font=\footnotesize] {
 0\& 0\& 0\& 0\& 0\& 0\& 0\& 0\& 0\\
 0\& 0\& 0\& 0\& 0\& 0\& 0\& 0\& 0\\
 0\& 0\& 0\&\!-1\!\& 0\& 0\& 0\& 0\& 0\\
 0\& 0\& 0\& 0\& 0\& 0\& 0\& 0\& 0\\
 0\& 0\&\!-1\!\& 0\& 0\&\!-1\!\& 0\& 0\& 0\\
 0\& 0\& 0\& 0\& 0\& 0\& 0\& 0\& 0\\
 0\& 0\& 0\& 0\& 0\& 0\&\!-1\!\& 0\& 0\\
 0\& 0\& 0\& 0\& 0\&\!-1\!\& 0\& 0\& 0\\
 0\& 0\& 0\& 0\& 0\& 0\& 0\& 0\& 0\\
};
\node[left=4mm of A] {$Y =$};
\node[draw, thick, fit=(A-1-1)(A-1-9), inner sep=0pt] {};
\node[draw, thick, fit=(A-2-1)(A-2-9), inner sep=0pt] {};
\node[draw, thick, fit=(A-3-1)(A-3-3), inner sep=0pt] {};
\node[draw, thick, fit=(A-3-5)(A-3-9), inner sep=0pt] {};
\node[draw, thick, fit=(A-4-1)(A-4-9), inner sep=0pt] {};
\node[draw, thick, fit=(A-5-1)(A-5-2), inner sep=0pt] {};
\node[draw, thick, fit=(A-5-4)(A-5-5), inner sep=0pt] {};
\node[draw, thick, fit=(A-5-7)(A-5-9), inner sep=0pt] {};
\node[draw, thick, fit=(A-6-1)(A-6-9), inner sep=0pt] {};
\node[draw, thick, fit=(A-7-1)(A-7-6), inner sep=0pt] {};
\node[draw, thick, fit=(A-7-8)(A-7-9), inner sep=0pt] {};
\node[draw, thick, fit=(A-8-1)(A-8-5), inner sep=0pt] {};
\node[draw, thick, fit=(A-8-7)(A-8-9), inner sep=0pt] {};
\node[draw, thick, fit=(A-9-1)(A-9-9), inner sep=0pt] {};
\end{tikzpicture}
\]\end{example}

Let $X \in\X_n$ be a $(0,k)$-matrix for $k \in \{-1,1\}$. Define $\R$ and $\Ss$ to be the sets respectively containing the type 1 row regions and the type 1 column regions. If a single $-k$ can be added to each region in $\R$ such that there is also a single $-k$ in each region in $\Ss$, then the resulting matrix is an ASM. Clearly this is not possible if any region in $\R$ or $\Ss$ is empty, or if $|\R| \ne |\Ss|$. We say that a set of row regions and column regions has an \emph{overlap} if there is a row region and a column region among the sets with an entry in common. We now state our Frobenius-K\"onig type result for ASMs.

\begin{theorem}\label{f-k for ASMs}
    Let $X\in\X_n$ be a $(0,k)$-matrix for $k\in\{1,-1\}$ with respective type 1 row and column regions $\R$ and $\Ss$. Suppose no region in $\R$ or $\Ss$ is empty, and $\R = \Ss = m$. Then there exists an $n \times n$ ASM $A$ such that $A(k) = X$ if and only if there is no non-overlapping set of $r$ regions in $\R$ and $s$ regions in $\Ss$ with $r+s = m+1$. 
\end{theorem}

\begin{proof}
    There exists such an $A$ if and only if 0 entries can be replaced with $-k$ such that there is exactly one $-k$ in each row region of $\R$ and in each column region of $\Ss$.

    Number the regions in $\R$ and $\Ss$ each with $1, 2, \dots, m$, and construct the $m \times m$ $(0,1)$-matrix $Q$ with $Q_{ij} = 1$ if and only if $\R_i$ has an entry in common with $\Ss_j$. By the Frobenius-K\"onig Theorem, $Q$ contains an $m \times m$ permutation matrix $P$ if and only if $Q$ has no $r \times s$ zero submatrix with $r+s = m+1$. An $r \times s$ zero submatrix of $Q$ corresponds to a set of $r$ regions in $\R$ and $s$ regions in $\Ss$ with no overlap.
    
    Finding such a $P$ allows us to construct the required ASM $A$ as follows. Replace the 0 entry at the intersection of $\R_i$ and $\Ss_j$ if and only if $P_{ij} = 1$. Since $P$ is a permutation matrix, this means there is exactly one replacement made in each region of $\R$ and each region of $\Ss$. We can therefore reinterpret the Frobenius-K\"onig Theorem here to say that such an $A$ exists if and only if there is no non-overlapping set of $r$ regions in $\R$ and $s$ regions in $\Ss$ with $r+s=m+1$.
\end{proof}

Note that the number of such ASMs $A$ satisfying $A(k) = X$ is equal to the permanent of the matrix $Q$ defined in the proof of Theorem \ref{f-k for ASMs}.

Consider~$X$ from Example~\ref{regions}. Numbering the type~1 regions in order from left to right, then down, gives the following~$Q$ and a suitable ASM~$A$. The~$-1$ entries of~$A$ and corresponding non-zero entries of~$P\le Q$ are indicated in red.
\[ {\footnotesize \arraycolsep=3pt \begin{pmatrix}
    0 & {\color{red}1} & 0 & 0 & 0 \\
    {\color{red}1} & 1 & 1 & 0 & 0 \\
    0 & 0 & 0 & 1 & {\color{red}1} \\
    0 & 0 & 1 & {\color{red}1} & 0 \\
    1 & 0 & {\color{red}1} & 1 & 1
\end{pmatrix} , \qquad
\begin{pmatrix}
 & & & +& & & & & & & & \\
 & & +&{\color{red}-}& & & +& & & & & \\
 & & & & +& & & & & & & \\
 & & & & & & & +& & & & \\
 & & & & & & & & & +& & \\
 +& &{\color{red}-}& & & +& & & &{\color{red}-}& & +\\
 & & & +& && & {\color{red}-}& +& & & \\
 & +& & &{\color{red}-}& & & & & & +& \\
 & & +& & & & & & & & & \\
 & & & & & & & +& & & & \\
 & & & & & & & & & +& & \\
 & & & & +& & & & & & &
\end{pmatrix}} \]

The type~1 regions of~$Y$ from Example~\ref{regions} gives the following~$Q$, with positive entries of permutation~$P\le Q$ and $A$ indicated in red.
\[ {\footnotesize \arraycolsep=3pt \begin{pmatrix}
    1 & 1 & 1 & 0 & {\color{red}1} & 0 & 1 & 1 & 0 & 0 & 1 & 0 & 1 & 1 \\
    1 & {\color{red}1} & 1 & 0 & 1 & 0 & 1 & 1 & 0 & 0 & 1 & 0 & 1 & 1 \\
    1 & 1 & {\color{red}1} & 0 & 0 & 0 & 0 & 0 & 0 & 0 & 0 & 0 & 0 & 0 \\
    0 & 0 & 0 & 0 & 0 & 0 & {\color{red}1} & 1 & 0 & 0 & 1 & 0 & 1 & 1 \\
    1 & 1 & 1 & 0 & 0 & 1 & 1 & {\color{red}1} & 0 & 0 & 1 & 0 & 1 & 1 \\
    {\color{red}1} & 1 & 0 & 0 & 0 & 0 & 0 & 0 & 0 & 0 & 0 & 0 & 0 & 0 \\
    0 & 0 & 0 & 0 & 0 & {\color{red}1} & 1 & 0 & 0 & 0 & 0 & 0 & 0 & 0 \\
    0 & 0 & 0 & 0 & 0 & 0 & 0 & 0 & 0 & 0 & {\color{red}1} & 0 & 1 & 1 \\
    1 & 1 & 0 & 1 & 0 & 1 & 1 & 0 & 1 & 0 & 1 & 0 & {\color{red}1} & 1 \\
    1 & 1 & 0 & 1 & 0 & 1 & 0 & 0 & {\color{red}1} & 0 & 0 & 0 & 0 & 0 \\
    0 & 0 & 0 & 0 & 0 & 0 & 0 & 0 & 0 & 0 & 0 & 0 & 1 & {\color{red}1} \\
    1 & 1 & 0 & {\color{red}1} & 0 & 1 & 1 & 0 & 0 & 0 & 0 & 0 & 0 & 0 \\
    0 & 0 & 0 & 0 & 0 & 0 & 0 & 0 & 0 & 0 & 0 & {\color{red}1} & 1 & 1 \\
    1 & 1 & 0 & 1 & 0 & 1 & 1 & 0 & 0 & {\color{red}1} & 0 & 1 & 1 & 1
\end{pmatrix} , \qquad \begin{pmatrix}
 & & & {\color{red}+}& & & & & \\
 & {\color{red}+}& & & & & & & \\
 & & {\color{red}+}&-& {\color{red}+}& & & & \\
 & & & & & {\color{red}+}& & & \\
 {\color{red}+}& &-& {\color{red}+}& &-& {\color{red}+}& & \\
 & & & & & & & {\color{red}+}& \\
 & & & & & {\color{red}+}&-& & {\color{red}+}\\
 & & {\color{red}+}& & &-& {\color{red}+}& & \\
 & & & & & {\color{red}+}& & & 
\end{pmatrix}} \]

We now consider the same Frobenius-K\"onig type problem for the wider class $\W_n$ consisting of all matrices in $\X_n$ in which the non-zero entries of each row/column alternate in sign, and the sum of each row/column is in $\{0,\pm1\}$. Note that ASMs are exactly the elements of $\W_n$ with row-sum vector $R = (1,1,\dots,1)$ and column-sum vector $S = (1,1,\dots,1)$.

In addition to type $k$ and type $-k$ regions, we define the set of \emph{type 0} regions, which contains each type $k$ region, and the union of the first and last row/column region if they are disjoint.
\begin{example}\label{type 0 for Y}
Below is the $(0,-1)$-matrix $Y \in \X_9$ from Example \ref{type 1 Y}, instead with the type 0 row and column regions indicated. To obtain a matrix in $\W_9$ with all row sums 0, a single $+1$ must be added to each type 0 row region.
\[
\begin{tikzpicture}
\matrix (A) [matrix of math nodes,
             left delimiter=(,
             right delimiter=),
             ampersand replacement=\&,
             row sep=0.2em,
             column sep=0.2em,
             font=\footnotesize] {
 0\& 0\& 0\& 0\& 0\& 0\& 0\& 0\& 0\\
 0\& 0\& 0\& 0\& 0\& 0\& 0\& 0\& 0\\
 0\& 0\& 0\&\!-1\!\& 0\& 0\& 0\& 0\& 0\\
 0\& 0\& 0\& 0\& 0\& 0\& 0\& 0\& 0\\
 0\& 0\&\!-1\!\& 0\& 0\&\!-1\!\& 0\& 0\& 0\\
 0\& 0\& 0\& 0\& 0\& 0\& 0\& 0\& 0\\
 0\& 0\& 0\& 0\& 0\& 0\&\!-1\!\& 0\& 0\\
 0\& 0\& 0\& 0\& 0\&\!-1\!\& 0\& 0\& 0\\
 0\& 0\& 0\& 0\& 0\& 0\& 0\& 0\& 0\\
};
\node[left=4mm of A] {$Y =$};

\newcommand{\fullbox}[4]{
  \draw[thick] (A-#1-#2.north west) rectangle (A-#3-#4.south east);
}

\newcommand{\leftopen}[4]{
  \draw[thick] (A-#1-#2.north west) -- (A-#3-#4.north east);
  \draw[thick] (A-#1-#2.south west) -- (A-#3-#4.south east);
  \draw[thick] (A-#3-#4.north east) -- (A-#3-#4.south east);
}

\newcommand{\rightopen}[4]{
  \draw[thick] (A-#1-#2.north west) -- (A-#3-#4.north east);
  \draw[thick] (A-#1-#2.south west) -- (A-#3-#4.south east);
  \draw[thick] (A-#1-#2.north west) -- (A-#1-#2.south west);
}

\newcommand{\topopen}[4]{
  \draw[thick] (A-#1-#2.north west) -- (A-#3-#4.south west);
  \draw[thick] (A-#1-#2.north east) -- (A-#3-#4.south east);
  \draw[thick] (A-#3-#4.south west) -- (A-#3-#4.south east);
}

\newcommand{\bottomopen}[4]{
  \draw[thick] (A-#1-#2.north west) -- (A-#1-#2.north east);
  \draw[thick] (A-#1-#2.north west) -- (A-#3-#4.south west);
  \draw[thick] (A-#1-#2.north east) -- (A-#3-#4.south east);
}

\leftopen  {3}{1}{3}{3}
\rightopen {3}{5}{3}{9}
\leftopen  {5}{1}{5}{2}
\fullbox   {5}{4}{5}{5}
\rightopen {5}{7}{5}{9}
\leftopen  {7}{1}{7}{6}
\rightopen {7}{8}{7}{9}
\leftopen  {8}{1}{8}{5}
\rightopen {8}{7}{8}{9}

\topopen    {1}{3}{4}{3}
\bottomopen {6}{3}{9}{3}

\topopen    {1}{4}{2}{4}
\bottomopen {4}{4}{9}{4}

\topopen    {1}{6}{4}{6}
\fullbox    {6}{6}{7}{6}
\bottomopen {9}{6}{9}{6}

\topopen    {1}{7}{6}{7}
\bottomopen {8}{7}{9}{7}

\end{tikzpicture}
\]\end{example}

Let $X \in\X_n$ be a $(0,k)$-matrix for $k \in \{-1,1\}$, and let $R$ and $S$ be vectors with entries in $\{0,\pm1\}$. Now define $\R$ and $\Ss$ to be the sets respectively containing the type $R_i$ regions in row $i$ and the type $S_j$ regions in column $j$ of $X$, for $i,j \in [n]$. If a single $-k$ can be added to each region in $\R$ such that there is also a single $-k$ in each region in $\Ss$, then the resulting matrix is in $\W_n$, such that row $i$ sums to $R_i$ for all $i \in [n]$, and column $j$ sums to $S_i$ for all $j \in [n]$. Clearly this is not possible if any region in $\R$ or $\Ss$ is empty, if $\sum R_i \ne \sum S_j $, or if $|\R| \ne |\Ss|$. We now state our more general Frobenius-K\"onig type result for matrices in $\W_n$.

\begin{corollary}\label{f-k for D_n}
    Let $X \in\X_n$ be a $(0,k)$-matrix for $k \in \{-1,1\}$, $R$ and $S$ be vectors with entries in $\{0,\pm1\}$, and $\R$ and $\Ss$ to be the sets respectively containing the type $R_i$ regions in row $i$ and the type $S_j$ regions in column $j$ of $X$, for $i,j \in [n]$. Suppose no region in $\R$ or $\Ss$ is empty, $\sum R_i = \sum S_j$, and $|\R| = |\Ss| = m$. Then there exists $A \in \W_n$ with row sums $R$ and column sums $S$ such that $A(k) = X$ if and only if there is no non-overlapping set of $r$ regions in $\R$ and $s$ regions in $\Ss$ with $r+s = m+1$. 
\end{corollary}

\begin{proof}
    The proof is identical to that of Theorem \ref{f-k for ASMs}, simply replacing ``ASM $A$" with ``$A \in \W_n$".
\end{proof}

The type 0 regions of $Y$ from Example \ref{type 0 for Y} give the following $Q$, with positive entries of permutation $P\le Q$ and $A$ indicated in red. While the non-zero entries in each row and column of $A$ alternate in sign, the last non-zero in each row and column is the opposite to the first. Since all row and column sums are 0, the alternating condition continues cycling around each row and column, and such matrices could therefore be considered \emph{toroidal ASMs}.
\[ Q = \arraycolsep=4pt \begin{pmatrix}
    1 & 0 & {\color{red}1} & 0 & 1 \\
    0 & 0 & 0 & 0 & {\color{red}1} \\
    0 & {\color{red}1} & 0 & 0 & 0 \\
    1 & 1 & 0 & {\color{red}1} & 0 \\
    {\color{red}1} & 1 & 0 & 0 & 1
\end{pmatrix} , \quad
A = \begin{pmatrix}
 & & & & & & & & \\
 & & & & & & & & \\
 & & &-& & {\color{red}+}& & & \\
 & & & & & & & & \\
 & &-& {\color{red}+}& &-& {\color{red}+}& & \\
 & & & & & & & & \\
 & & & & & {\color{red}+}&-& & \\
 & & {\color{red}+}& & &-& & & \\
 & & & & & & & &
\end{pmatrix} \]

ASMs and toroidal ASMs have all row and column sums equal to one another, but we may also apply Corollary \ref{f-k for D_n} to find matrices with varying row and column sums. Similar, but not equivalent, generalisations of ASMs have previously been studied by Brualdi and Kim \cite{BrualdiKim14}.

\begin{example}
    Suppose we have the following $X \in \X_4$, and require row sums $R = (0, -1, +1, -1)$ and column sums $S = (0, -1, 0, 0)$. This means that we need to identify the type 0 regions in row 1 and columns 1, 3, and 4, type 1 regions in row 3, and type -1 regions in column 2, and rows 2 and 4.
    \[\begin{tikzpicture}

\newcommand{\fullbox}[4]{
  \draw[thick] (A-#1-#2.north west) rectangle (A-#3-#4.south east);
}

\newcommand{\leftopen}[4]{
  \draw[thick] ([xshift=-5pt]A-#1-#2.north west) -- (A-#1-#4.north east);
  \draw[thick] ([xshift=-5pt]A-#3-#2.south west) -- (A-#3-#4.south east);
  \draw[thick] (A-#1-#4.north east) -- (A-#3-#4.south east);
}

\newcommand{\rightopen}[4]{
  \draw[thick] (A-#1-#2.north west) -- ([xshift=5pt]A-#1-#4.north east);
  \draw[thick] (A-#3-#2.south west) -- ([xshift=5pt]A-#3-#4.south east);
  \draw[thick] (A-#1-#2.north west) -- (A-#3-#2.south west);
}

\newcommand{\topopen}[4]{
  \draw[thick] ([yshift=5pt]A-#1-#2.north west) -- (A-#3-#2.south west);
  \draw[thick] ([yshift=5pt]A-#1-#4.north east) -- (A-#3-#4.south east);
  \draw[thick] (A-#3-#2.south west) -- (A-#3-#4.south east);
}

\newcommand{\bottomopen}[4]{
  \draw[thick] (A-#1-#2.north west) -- (A-#1-#4.north east);
  \draw[thick] (A-#1-#2.north west) -- ([yshift=-5pt]A-#3-#2.south west);
  \draw[thick] (A-#1-#4.north east) -- ([yshift=-5pt]A-#3-#4.south east);
}

\matrix (A) [matrix of math nodes,
             left delimiter=(,
             right delimiter=),
             ampersand replacement=\&,
             row sep=0.3em,
             column sep=0.3em,
             font=\small] {
 0 \& \!-1\! \& 0 \& 0 \\
 0 \& 0 \& 0 \& \!-1\! \\
 0 \& 0 \& \!-1\! \& 0 \\
 \!-1\! \& 0 \& 0 \& \!-1\! \\
};

\node[left=4mm of A] {$X =$};

\leftopen   {1}{1}{1}{1}
\rightopen  {1}{3}{1}{4}

\fullbox    {3}{1}{3}{2}

\draw[thick]
  ([xshift=-7pt,yshift=-5pt]A-3-4.center)
  rectangle
  ([xshift= 7pt,yshift= 5pt]A-3-4.center);

\draw[thick]
  ([xshift=-5pt,yshift=-7pt]A-3-4.center)
  rectangle
  ([xshift= 5pt,yshift= 7pt]A-3-4.center);
\fullbox    {4}{2}{4}{3}

\fullbox    {1}{1}{3}{1}
\topopen    {1}{3}{2}{3}

\draw[thick]
  ([xshift=-5pt,yshift=-7pt]A-1-4.center)
  rectangle
  ([xshift= 5pt,yshift= 7pt]A-1-4.center);

\bottomopen {4}{3}{4}{3}

\end{tikzpicture}\]
This gives the following $Q$, and $A$ with row sums $R$ and column sums $S$. Positive entries of $P \le Q$ and $A$ are indicated in red.
\[ Q = \arraycolsep=4pt \begin{pmatrix}
    1 & 1 & {\color{red}1} & 0 \\
    {\color{red}1} & 0 & 0 & 0 \\
    0 & 0 & 0 & {\color{red}1} \\
    0 & {\color{red}1} & 0 & 0
\end{pmatrix} , \quad
A = \begin{pmatrix}
      & - &   & {\color{red}+} \\
      &   &   & - \\ 
    {\color{red}+}&   & - & {\color{red}+} \\ 
    - &   & {\color{red}+} & - \\ 
\end{pmatrix}\]
\end{example}

\section{Transversals}\label{sec:transversals}

A \emph{transversal} of a Latin square of order~$n$ is a set of~$n$ cells with one cell in each row, one in each column, such that no two cells contain the same symbol.

It is well-known that there exist Latin squares of every even order possessing no transversals. It is conjectured that every Latin square of odd order possesses a transversal. We refer to \cite{wanless} for a survey.

Finding a Latin square~$L'$ of order $n$ which is orthogonal to some given Latin square~$L$ of order $n$ is equivalent to decomposing~$L$ into a set of~$n$ disjoint transversals $T_1, \dots, T_n$.  The cells of~$L'$ which correspond to~$T_i$ in~$L$ can all be filled with symbol~$i$, resulting in $L'$ being orthogonal to~$L$.

Below, in Figure \ref{fig:LatinSqu_transversals} is a Latin square~$L$ of order~4, with four disjoint transversals $T_1, T_2, T_3, T_4$ indicated in black, blue, red, and green, respectively.  We then produce the orthogonal Latin square~$L'$ in Figure \ref{fig:LatinSqu_transversals} (right) with symbol~$i$ corresponding to~$T_i$.

\begin{figure}[h!] \centering
  $ L = \ \begin{array}{|c|c|c|c|}\hline
      1 & {\color{b} 2} & {\color{red} 3 } & {\color{newgreen} 4 } \\\hline
      {\color{b} 3} & 4 & {\color{newgreen} 1} & {\color{red} 2} \\\hline
      {\color{red} 4} & {\color{newgreen} 3} & 2 & {\color{b} 1} \\\hline
      {\color{newgreen} 2} & {\color{red} 1} & {\color{b} 4} & 3 \\\hline
      \end{array} $ \qquad\qquad
  $ L' = \ \begin{array}{|c|c|c|c|}\hline
      1 & 2 & 3 & 4 \\\hline
      2 & 1 & 4 & 3 \\\hline
      3 & 4 & 1 & 2 \\\hline
      4 & 3 & 2 & 1 \\\hline
      \end{array} $

    \captionsetup{width=0.85\textwidth}
    \caption{Latin square $L$ with transversals $T_1$, \textcolor{b}{$T_2$}, \textcolor{red}{$T_3$}, \textcolor{ng}{$T_4$} (left) and the orthogonal $L'$ obtained from these transversals (right).}
    \label{fig:LatinSqu_transversals}
\end{figure}

Recall from Section \ref{subsec:alternsignmat} that $\langle A, B \rangle$ denotes the Frobenius inner product of the matrices $A$ and $B$. Finding a transversal of a Latin square~$L$ is equivalent to finding a permutation matrix~$P$ for which $\langle P, L_k \rangle = 1$ for all planes~$L_k$ of~$L$.  For example, each of the following permutation matrices implies a transversal in the example above (Figure \ref{fig:LatinSqu_transversals}).
\[ \begin{pmatrix}
     + & & & \\
     & + & & \\
     & & + & \\
     & & & + \\
\end{pmatrix},
\begin{pmatrix}
     & + & & \\
     + & & & \\
     & & & + \\
     & & + & \\
\end{pmatrix},
\begin{pmatrix}
     & & + & \\
     & & & + \\
     + & & & \\
     & + & & \\
\end{pmatrix},
\begin{pmatrix}
     & & & + \\
     & & + & \\
     & + & & \\
     + & & & \\
\end{pmatrix} \]

We now generalise this idea to Italian squares.
\begin{definition}
    A \emph{transversal} of an Italian square is an ASM~$A$ for which $\langle A, L_k \rangle = 1$ for each plane~$L_k$ of~$L$.
\end{definition}

\begin{example}
    Transversals of the following Italian square are indicated in blue and red, with blue representing positive cells and red negative. The first corresponds to a permutation, and the second an ASM.
    \[ \begin{array}{|c|c|c|c|c|}\hline
         {\color{b} 1} & 2 & 3 & 4 & 5 \\\hline
         2 & \SmallArray{1-2\\+3} & \SmallArray{2-3\\+4} & \SmallArray{3-4\\+5} & {\color{b}4} \\\hline
         3 & \SmallArray{2-3\\+4} & \SmallArray{1-2+3\\-4+5} & {\color{b} \SmallArray{2-3\\+4}} & 3 \\\hline
         4 & {\color{b} \SmallArray{3-4\\+5}} & \SmallArray{2-3\\+4} & \SmallArray{1-2\\+3}& 2 \\\hline
         5 & 4 & {\color{b} 3} & 2 & 1 \\\hline
    \end{array} \qquad
    \begin{array}{|c|c|c|c|c|}\hline
         1 & 2 & 3 & {\color{b} 4} & 5 \\\hline
         2 & {\color{b} \SmallArray{1-2\\+3}} & \SmallArray{2-3\\+4} & \SmallArray{3-4\\+5}& 4 \\\hline
         3 & \SmallArray{2-3\\+4} & {\color{b} \SmallArray{1-2+3\\-4+5}} & \SmallArray{2-3\\+4} & 3 \\\hline
         {\color{b} 4} & \SmallArray{3-4\\+5} & \SmallArray{2-3\\+4}& {\color{red} \SmallArray{1-2\\+3}} & {\color{b} 2} \\\hline
         5 & 4 & 3 & {\color{b} 2} & 1 \\\hline
    \end{array} \]
\end{example}

\begin{example}
    Exhaustive computer search shows that the following is the only Italian square (up to rearrangement of its planes) with either $4 \times 4$ diamond ASM as a transversal.  Both are transversals, with positive transversal cells indicated in blue, and negatives in red.
    \[ \begin{array}{|c|c|c|c|}\hline
         1 & {\color{b} 2} & 3 & 4 \\\hline
         {\color{b} 2} & {\color{red} 1 } & {\color{b} 4} & 3 \\\hline
         4 & {\color{b} 3} & {\color{red} 2} & {\color{b} 1} \\\hline
         3 & 4 & {\color{b} 1} & 2 \\\hline
    \end{array} \qquad\qquad
    \begin{array}{|c|c|c|c|}\hline
        1 & 2 & {\color{b} 3} & 4 \\\hline
        2 & {\color{b} 1} & {\color{red} 4} & {\color{b} 3} \\\hline
        {\color{b} 4} & {\color{red} 3} & {\color{b} 2} & 1 \\\hline
        3 & {\color{b} 4} & 1 & 2 \\\hline
    \end{array} \]
\end{example}

Suppose~$L$ is an Italian square of order~$n$.  Then any transversal ASM~$A$ of~$L$ must be perpendicular to each plane of~$L$; that is, $\langle L_i, A \rangle = 1$ for all planes~$L_i$ of~$L$.
If~$A$ is an ASM, then we must also have that $\langle \frac 1 n J, A \rangle = 1$, where~$J$ is the all-ones matrix, and that~$A$ has constant row and column sums. 

\begin{definition}
    Within the space of all rational $n \times n$ matrices, let $\RC$ denote the subspace for which there exists $\lambda \in \mathbb Q$ such that all row and column sums are equal to $\lambda$. 
    
    Let $\RC_\lambda$ denote the subset of $\RC$ of elements with constant row and column sums $\lambda$.  
\end{definition}

Observe that $\RC=\mathrm{span}\{\RC_0,J \}$, and $\RC_0$ is a subspace of dimension $(n-1)^2$, since
\[
\RC_0 = \mathrm{span}\{ T_{i,j;n,n}:i,j\in \{1,\ldots,n-1\}\},
\]
where $T_{i,j;n,n}$ are the $T$-blocks introduced earlier in Section \ref{sec:mate}.

\begin{lemma}\label{lem:zeroportion}
    Every ASM $A$ can be written as $A = A^0 + \frac 1 n J$, where $A^0 \in \RC_0$.
\end{lemma}

\begin{proof}
As $A$ is an ASM with entries $A_{ij}$, we have $\sum_{i}A_{ij}= 1$ for all $j$, which implies that the column sums satisfy
\[ \sum_{i} (A - \tfrac 1 n J)_{ij}= \sum_{i} A_{ij} - \tfrac 1 n \sum_{i} J_{ij} = 0 . \]
Similarly, with $\sum_{j}A_{ij}= 1$ for all $i$, it follows that also all row sums of $A -\frac 1 n J$ are equal to $0$.
\end{proof}

\begin{definition}
    Given an ASM~$A$, we call $A^0 := A - \frac 1 n J$ the \emph{zero portion} of~$A$.
\end{definition}

Lemma \ref{lem:zeroportion} allows us a change of perspective in characterising pairs of orthogonal ASMs, and with that of Italian squares, in terms of their zero portions. 

\begin{lemma} \label{lem:char_orth_ASMs_zeroportion}
    Let $A$ and $B$ be two ASMs of the same order, with zero portions $A^0$ and $B^0$, respectively. Then \begin{align} \label{eq:equivalence_innerprod_zeroportion}
    \langle A, B \rangle = 1 ~\text{ is equivalent to }~ \langle A^0, B^0 \rangle = 0 .
\end{align}  
\end{lemma}

\begin{proof}
This is a direct consequence of the equation $\langle A^0,B^0 \rangle = \langle A, B \rangle -1$, which itself follows from the definitions of an ASM and its zero portion.
Indeed, we have $\langle A, B \rangle = \langle A^0 + \frac 1 n J, B^0 + \frac 1 n J \rangle = \langle A^0, B^0 \rangle + \frac 1 n (\langle A^0, J \rangle + \langle J, B^0 \rangle) + \frac 1 {n^2} \langle J, J \rangle = \langle A^0, B^0 \rangle + 0 + 1$.
\end{proof}

\begin{corollary}
    An ASM $A$ is a transversal of the Italian square $L$ if and only if $\langle A^0, L_k^0 \rangle = 0$ for each plane~$L_k$ of~$L$.
\end{corollary}

We now consider the Italian square corresponding to the ASHM with most nonzero entries, the \emph{diamond} Italian square \textbf{$\mathcal D_n$}. We first show existence of a transversal of $\mathcal D_n$ for $n \equiv \pm 1 \bmod 6$, and subsequently apply the above concepts to show nonexistence for $n \not\equiv \pm 1 \bmod 6$.

Brualdi and Dahl define the diamond ASHM/Italian square $\mathcal D_n$ of order~$n$~\cite[Ex.~3]{BrualdiDahl18}.  The following describes the $(i, j, k)$-entry, indexed from~$1$ to~$n$:
\[ (\mathcal D_n)_{ijk} = \begin{cases}
    (-1)^{i+j+k+1} & \text{if } |i \!+\! j \!-\! n \!-\! 1| \le n \!-\! k \text{ and } |i \!-\! j| \le k \!-\! 1 \\
    0 & \text{otherwise}
\end{cases} \]

We denote the planes of $\mathcal D_n$ by $E_k$; that is, $\mathcal D_n = (E_1, \ldots, E_n)$.  See Figure \ref{fig:Diamond_4} for an example when $n = 4$.

\begin{figure}[h]
\[ \begin{pmatrix}
     + & & & \\
     & + & & \\
     & & + & \\
     & & & + \\
\end{pmatrix},
\begin{pmatrix}
     & + & & \\
     + & - & + & \\
     & + & - & + \\
     & & + & \\
\end{pmatrix},
\begin{pmatrix}
     & & + & \\
     & + & - & + \\
     + & - & + & \\
     & + & & \\
\end{pmatrix},
\begin{pmatrix}
     & & & + \\
     & & + & \\
     & + & & \\
     + & & & \\
\end{pmatrix} \]
\captionsetup{width=0.98\textwidth}
\caption{The planes $E_1,E_2,E_3, E_4$ of the diamond Italian square $\mathcal D_4$.}
\label{fig:Diamond_4}
\end{figure}

\begin{lemma}\label{lem: diamond transversal} For every $n \ge 1$ satisfying $n \equiv \pm 1 \bmod 6$, the diamond~$\mathcal D_n$ has a transversal permutation matrix~$P$ given by the non-zero entries
  \[ \begin{cases} (2 \ell \!-\! 1, 4 \ell \!-\! 2) \text{ and } (2 \ell, 4 \ell \!-\! 1) & \text{for } \ell = 1, \dots, \lceil \frac {n - 1} 4 \rceil \,, \\
       (n \!-\! 2 \ell \!+\! 1, 4 \ell) \text{ and } (n \!-\! 2 \ell, 4 \ell + 1) & \text{for } \ell = 1, \dots, \lfloor \frac {n - 1} 4 \rfloor \,, \end{cases} \]
  as well as $(n, 1)$. \end{lemma}

\begin{proof} We show that $\langle E_k, P \rangle = 1$ for each of the planes $E_k$ of the diamond~$\mathcal D_n$.

  For $k = 1$ the plane $E_1$ is the identity matrix.  It shares non-zero entries with~$P$ precisely if $n - 2 \ell + 1 = 4 \ell$ or $n - 2 \ell = 4 \ell + 1$ for some $\ell = 1, \dots, \lceil \frac {n - 1} 4 \rceil$, that is when $n = 6 \ell - 1$ or $n = 6 \ell + 1$.  Hence if $n \equiv \pm 1 \bmod 6$ we have $\langle E_1, P \rangle = 1$.

  Now consider for $k = 1, \dots, n \!-\! 1$ the matrix~$F_k$ defined by
  \[ F_k := E_k + E_{k+1} . \]
  We will show that $\langle F_k, P \rangle = 2$ for any~$k$. 
  From this it follows that $\langle E_{k+1}, P \rangle = \langle F_k, P \rangle - \langle E_k, P \rangle = 2 - \langle E_k, P \rangle$, and thus $\langle E_k, P \rangle = 1$ for all~$k$ by induction, proving our claim.

  The proof is based on the observation that~$F_k$ is a $(0, 1)$-matrix with non-zero entries $(i, j)$ precisely on the four lines where $j - i = \pm k$ and $i + j = n + 1 \pm k'$ with $k' := n - k$.  The first two lines are of diagonal type, the last two are of anti-diagonal type.  
  Also, they always contain two lines of odd type and two of even type (i.e.\ lines with $i \pm j = \text{odd or even}$).  Below is an example for $F_1$, $F_2$, $F_3$, $F_4$ if $n = 5$.
  \[ {\small\begin{pmatrix} 1 & 1 & & & \\ 1 & & 1 & & \\ & 1 & & 1 & \\ & & 1 & & 1 \\ & & & 1 & 1 \end{pmatrix}} ~
    {\small\begin{pmatrix} & 1 & 1 & & \\ 1 & & & 1 & \\ 1 & & & & 1 \\ & 1 & & & 1 \\ & & 1 & 1 & \end{pmatrix}} ~
    {\small\begin{pmatrix} & & 1 & 1 & \\ & 1 & & & 1 \\ 1 & & & & 1 \\ 1 & & & 1 & \\ & 1 & 1 & & \end{pmatrix}} ~
    {\small\begin{pmatrix} & & & 1 & 1 \\ & & 1 & & 1 \\ & 1 & & 1 & \\ 1 & & 1 & & \\ 1 & 1 & & & \end{pmatrix}} \]
  
  From this description, the matrix~$P$ in the statement should be thought of as a matrix obtained by selecting two entries from precisely one of the lines of each~$F_k$.  In particular, the entries of~$P$ can be divided as follows.
  \begin{itemize}
  \item[$\circ$] Upper pairs of the form $(2 \ell - 1 ,\, 4 \ell - 2)$ and $(2 \ell ,\, 4 \ell - 1)$ share the diagonal $j - i = 2\ell - 1$.  These are \emph{odd} diagonals: $j - i \in \{ 1, 3, 5, \dots \}$.
  \item[$\circ$] Lower pairs of the form $(n - 2 \ell + 1 ,\, 4 \ell)$ and $(n - 2 \ell ,\, 4 \ell + 1)$ share the anti-diagonal $i + j = n + 2 \ell + 1$.  These are \emph{even} anti-diagonals strictly below the centre: $i + j \in \{ n + 3, n + 5, \dots \}$.
  \item[$\circ$] Corner: The single entry $(n, 1)$ on the central anti-diagonal $i + j = n + 1$.
  \end{itemize}
  Here is an example for $n=7$, with upper pairs in blue, lower pairs in red, and the corner entry in green.
  \[ {\small P \,=\,
\begin{pmatrix}
  & {\color{b} 1} &  &  &  &  &  \\
  &  & {\color{b} 1} &  &  &  &  \\
  &  &  &  &  & {\color{b} 1} &  \\
  &  &  &  &  &  & {\color{b} 1} \\
  &  &  &  & {\color{red} 1} &  &  \\
  &  &  & {\color{red} 1} &  &  &  \\
  {\color{green!60!black} 1} &  &  &  &  &  &
\end{pmatrix} \qquad
\begin{array}{@{}l@{}}
\left.\vphantom{\begin{matrix}\\[0.8em]\end{matrix}}\right\}
  {\color{b} j-i = 1}\\[0.5em]
\left.\vphantom{\begin{matrix}\\[0.8em]\end{matrix}}\right\}
  {\color{b} j-i = 3}\\[0.3em]
\left.\vphantom{\begin{matrix}\\[0.8em]\end{matrix}}\right\}
  {\color{red} i+j = 10}\\[0.0em]
  \quad {\color{green!60!black} i+j = 8}
\end{array}} \]

The subsequent examination shows the following.  For \emph{odd} values of~$k$ one of the lines of~$F_k$ is always one of the lines of~$P$, and therefore shares~$2$ non-zero entries.
For \emph{even} values of~$k$, exactly two of the four lines of~$F_k$ contain entries of~$P$.
Each such active line meets a pair of~$P$ of the opposite type: an anti-diagonal of~$F_k$ crosses a diagonal of an upper pair, and a diagonal of~$F_k$ crosses an anti-diagonal of a lower pair or the corner.

  We proceed with the examination of shared ones of~$F_k$ with~$P$ in more detail.  For the upper pairs $(2 \ell - 1, 4 \ell - 2)$ and $(2 \ell, 4 \ell - 1)$ we have $j - i = 2 \ell - 1 = k$, and for the lower pairs $(n - 2 \ell + 1, 4 \ell)$ and $(n - 2 \ell, 4 \ell + 1)$ we find $i + j = n + 2 \ell + 1 = n + 1 + k'$, that is $k = n - 2 \ell$.  This shows that for all odd~$k$ we have exactly~$2$ shared ones.

  For the upper pairs $(2 \ell - 1, 4 \ell - 2)$ and $(2 \ell, 4 \ell - 1)$ we further have
  \[ i + j = 6 \ell - 3 \quad\text{ and }\quad i + j = 6 \ell - 1 \,, \]
  which is $n + 1 - k'$ (if $\ell \le \lfloor \frac {n+1} 6 \rfloor$) or $n + 1 + k'$ (if $\ell > \lfloor \frac {n+1} 6 \rfloor$) for some  $k' = n - k$.  And for the lower pairs $(n - 2 \ell + 1, 4 \ell)$ and $(n - 2 \ell, 4 \ell + 1)$, as well as the corner $(n, 1)$, we find
  \[ j - i = 6 \ell - n - 1 \quad\text{ and }\quad j - i = 6 \ell - n + 1 \,, \]
  which equals some~$-k$ (if $\ell \le \lfloor \frac {n-1} 6 \rfloor$) or~$k$ (if $\ell > \lfloor \frac {n-1} 6 \rfloor$).

  These affect only the even~$k$.  For the upper pairs, we find $k = 6 \ell - 4$ and $k = 6 \ell - 2$ if $\ell \le \lfloor \frac {n+1} 6 \rfloor$, or $k = 2 n - 6 \ell - 4$ and $k = 2 n - 6 \ell - 2$ otherwise.
  For the lower pairs and the corner we have $k = n - 6 \ell + 1$ and $k = n - 6 \ell - 1$ if $\ell \le \lfloor \frac {n-1} 6 \rfloor$, or $k = 6 \ell - n - 1$ and $k = 6 \ell - n + 1$.  
  Note that $2 n - 6 \ell - 4 \equiv 6 \ell - n - 1$ and $2 n - 6 \ell - 2 \equiv 6 \ell - n + 1$ modulo~$6$, as~$n$ is odd.
  
  We thus have a shared one for $k \bmod 6 \in \{ 2, 4 \}$ in the upper pairs if $\ell \le \lfloor \frac {n+1} 6 \rfloor$, for $k \bmod 6 \in \{ n \!+\! 1, n \!-\! 1 \}$ in the lower pairs or corner if $\ell \le \lfloor \frac {n-1} 6 \rfloor$, and for $k \bmod 6 \in \{ 2 n \!+\! 2, 2 n \!+\! 4 \}$ in the remaining cases.
  Since $n \equiv \pm 1 \bmod 6$ we have exactly~$2$ shared ones in total for all even~$k$.

  Altogether this shows $\langle F_k, P \rangle = 2$ for all~$k$ and concludes the proof. 
\end{proof}

\begin{theorem}\label{thm: diamond no transversal}
    The diamond Italian square\ $\mathcal D_n$ possesses a transversal if and only if $n \not\equiv \pm 1 \bmod 6$.
\end{theorem}

\begin{proof}
Lemma \ref{lem: diamond transversal} gives a construction for a transversal of $\D_n$ for all $n \equiv \pm 1 \bmod 6$. We now prove that this is a necessary condition on $n$.

Suppose $A$ is a transversal of $\mathcal D_n$. Then $A\in \RC$, and $\langle E_k^0,A\rangle=0$ for all $k\in[n]$. The set of matrices satisfying these two properties is a subspace of dimension $(n-1)(n-2)+1$. We proceed by finding a basis of this space, and showing that it cannot contain an ASM when $n\not\equiv \pm 1 \bmod 6$.

First consider matrices of the form $X_{ij,a} := T_{i,j;n,n} + (e_{n-1} \!-\! e_n)^\top a$ for $i \in \{ 1, \dots, n \!-\! 2 \}$, $j \in \{ 1, \dots, n \!-\! 1 \}$, where~$a$ is a vector of length~$n$ with $\sum a_i = 0$. That is, we define
\[
\bordermatrix{
        &        &        & j      &        & n      \cr
        & 0      & \cdots & 0      & \cdots & 0      \cr
        & \vdots & \ddots & \vdots & \ddots & \vdots \cr
i       & 0      & \cdots & 1      & \cdots & -1     \cr
        & \vdots & \ddots & \vdots & \ddots & \vdots \cr
n\!-\!1 & 0      & \cdots & 0      & \cdots & 0      \cr
n       & 0      & \cdots & -1     & \cdots & 1      \cr
}
\,+
\bordermatrix{
        &        &        &        &        &        \cr
        & 0      & \cdots & 0      & \cdots & 0      \cr
        & \vdots & \ddots & \vdots & \ddots & \vdots \cr
        & 0      & \cdots & 0      & \cdots & 0      \cr
        & \vdots & \ddots & \vdots & \ddots & \vdots \cr
        & a_1    & \cdots & a_j    & \cdots & a_n    \cr
        & -a_1   & \cdots & -a_j   & \cdots & -a_n   \cr
} =: X_{ij,a} .
\]
Any such matrix is in $\RC_0$, and any two with different indices $ij$ are linearly independent. We claim that there exists a unique solution $a$ for each $i,j$, and that the entries of $a$ are in fact integers.

By imposing $ \langle E_k^0, X_{ij, a} \rangle = 0$, we obtain the equations
\[ \begin{cases}
  a_{n-1} - a_n = b_1, &  \\
  a_{k-1} - 2a_k + a_{k+1} = b_k, & \text{for } 2 \le k \le n-1 \\
  -a_1 + a_2 = 
  b_n & 
\end{cases} \]
where 
\[ b_k := -(\mathcal D_n)_{ijk}+\delta_{i,n-k+1}+\delta_{j,n-k+1}-\delta_{k1} , \]
and $\delta_{xy}$ is the usual Kronecker delta symbol.

The final condition follows from the first $n-1$ conditions, and so the system is consistent. After taking into account that $\sum a_i=0$, we get that the system has a unique solution. Solving this recurrence, we get that

\[
a_t = \left( \sum_{s=n-t+1}^{n-1}\sum_{\ell=1}^s b_\ell \right) - \frac 1 n \sum_{k=1}^n \left( {n \choose 2} - {k \choose 2} \right) b_k .
\]

It is straightforward to see that $\sum_k b_k = 0$. Note that

\[
-\sum_{k=1}^n {k \choose 2} b_k = {n \!-\! i \!+\! 1 \choose 2} +{n \!-\! j \!+\! 1 \choose 2} - \sum_{k=1}^n {k \choose 2} (\mathcal D_n)_{ijk} ,
\]
and 
\[
\sum_{k=1}^n {k \choose 2} (\mathcal D_n)_{ijk} = \sum_{k=|i-j+1|}^{n-|i+j+n-1|} (-1)^{i+j+k+1} {k\choose 2} .
\]
We can use the elementary fact
\[ \sum_{k=a}^{a+2b} (-1)^{k-a} {k \choose 2} \,=\, {a \choose 2} + b (a + b) \,. \]

By symmetry we may assume $i \le j$.  In the case $i + j \le n + 1$ we let $a = j - i + 1$ and $b = i - 1$ to find
\begin{align*}
\sum_{k=j-i+1}^{j+i-1} (-1)^{i+j+k+1} {k \choose 2} &= {j \!-\! i \!+\! 1 \choose 2} + (i - 1) j \\
&= {i \choose 2} + {j \choose 2} 
\end{align*}
Similarly, if $i + j > n + 1$ we take $b = n - j$ and obtain
\begin{align*}
\sum_{k=j-i+1}^{2n-j-i+1} (-1)^{i+j+k+1} {k \choose 2} &= {j \!-\! i \!+\! 1 \choose 2} + (n - j) (n - i + 1) \\
&= {n \!+\! 1 \!-\! i \choose 2} + {n \!+\! 1 \!-\! j \choose 2} .
\end{align*}

Hence we obtain 
\[
a_t = \begin{cases}\left(\sum_{s=n-t+1}^{n-1}\sum_{\ell=1}^s b_\ell\right)+(n+1-i-j)&\text{if }i+j\leq n+1\\
\sum_{s=n-t+1}^{n-1}\sum_{\ell=1}^s b_\ell&\text{otherwise.}
\end{cases}
\]
Clearly these are integers for all $i,j,t$. We denote these unique solutions by $X_{ij}$. Next we consider a matrix of the following form, where again $a$ is a vector with $\sum a_i=0$.
\[
\bordermatrix{
        &        &        &       &        &       \cr
        & 0      & \cdots & 0      & \cdots & 1      \cr
        & \vdots & \ddots & \vdots & \ddots & \vdots \cr
       & 0      & \cdots & 0      & \cdots & 1     \cr
        & \vdots & \ddots & \vdots & \ddots & \vdots \cr
n-1     & 0      & \cdots & 0      & \cdots & 1      \cr
n       & 1      & \cdots & 1     & \cdots & 2-n      \cr
}
+
\bordermatrix{
        &        &        &       &        &       \cr
        & 0      & \cdots & 0      & \cdots & 0      \cr
        & \vdots & \ddots & \vdots & \ddots & \vdots \cr
       & 0      & \cdots & 0      & \cdots & 0      \cr
        & \vdots & \ddots & \vdots & \ddots & \vdots \cr
     & a_1    & \cdots & a_j    & \cdots & a_n    \cr
       & -a_1   & \cdots & -a_j   & \cdots & -a_n   \cr
}=:Y_a
\]
This matrix is in $\RC_1$, and hence is linearly independent from the $X_{ij}$'s. By imposing $\langle E_k^0,Y_a\rangle=0$, we get the system
\[
\begin{cases}
  a_{n-1} - a_n =n-1, &  \\
  a_{k-1} - 2a_k + a_{k+1} = -1, & \text{for } 2 \le k \le n-1, \\
  -a_1 + a_2 = 
  -1. & 
\end{cases}
\]

We can solve in a similar manner, obtaining the unique solution
\[
a_{t}= \frac{n^2-1}{6}-\frac{t(t-1)}{2}.
\]
We denote this unique solution by $Y$. Note that the entries of $Y$ are all integers if and only if $\frac{n^2-1}{6}$ is an integer, which occurs if and only if $n\equiv \pm1\bmod 6$.

Hence $\{Y,X_{ij}:i\in [n-2],j\in [n-1]\}$ forms a basis of the space of solutions. Moreover, any transversal ASM $A$ of $\mathcal D_n$ must be the sum of $Y$ and a $\{0,\pm 1\}$-sum of the $X_{ij}$, since otherwise the top right entry of $A$ would be negative, or the top row of $A$ would have two non-zero entries. Since each $X_{ij}$ is an integer matrix, and $Y$ is not an integer matrix if $n\not\equiv 1 \bmod 6$, it is impossible that $A$ is an integer matrix, and so cannot be an ASM.
\end{proof}

Note that this shows that a generalisation of the conjecture for transversals of Latin squares of odd order cannot be true for Italian squares; for example, $\mathcal D_9$ does not possess a transversal.

As a consequence of \Cref{thm: diamond no transversal}, we get the following result.

\begin{corollary}
The diamond Italian square $\mathcal D_n$ does not possess an orthogonal mate for $n \not\equiv\ \pm1 \bmod 6$.
\end{corollary}

\section{On the maximum cardinality of POIS}\label{sec:pois}

In this section we prove that no set of~$n$ pairwise orthogonal $n \times n$ Italian squares exists. 
We utilise the reinterpretation of the orthogonality of Italian squares in terms of the inner product of their zero portions.

\begin{definition}
For an Italian square~$L$ with planes $L_1, \dots, L_n$, we define the following subspaces of $RC$ and $RC_0$, respectively,
\begin{gather*}
    \RC(L) = \on{span} \{ L_1, \dots, L_n \} ,\\
    \RC_0(L) = \on{span} \{ L_1^0, \dots, L_n^0 \},
\end{gather*}
where $L_i^0$ denotes the zero portion of the plane $L_i$ as defined in Section \ref{sec:transversals}.    
\end{definition}

\begin{lemma} \label{lem:dim_italiansquare_zeroportion}
For an Italian square~$L$ of order $n$, it holds that $\dim(\RC(L)) = n$ and $\dim(\RC_0(L)) = n \!-\! 1$.
\end{lemma}

\begin{proof}
    Let $(L_i)$ be the planes of the Italian square~$L$. Since $\sum L_i = J$ and the first row of each $L_i$ contains precisely one non-zero entry, which is a~$1$, it follows that the coordinate of the non-zero entry of the first row of $L_i$ is different from the one of any other $L_j$.  Therefore no nontrivial linear combination of the $L_i$'s can give the zero matrix, proving the first claim.

As we can write each $L_i$ in terms of its zero portion, namely $L_i=L_i^0 + \frac{1}{n}J$, we have 
    \[ \RC(L)=\RC_0(L) + \operatorname{Span}\{J\}.\]
    Consequently, it holds that 
    \begin{align} \label{eq:dimensionformula_RCP}
    \dim (\RC(L)) = \dim (\RC_0(L) )+ 1 - \dim (\RC_0(L) \cap \operatorname{Span}\{J \}  ).
    \end{align}
    We claim that $\dim \left( \RC_0(L) \cap \on{Span}\{ J \} \right) = 0$.
    This follows from $\RC_0(L) \subseteq \RC_0$ and the fact that a matrix $\lambda J$ for $\lambda \in \mathbb Q$ has row (and column) sum zero only in case $\lambda = 0$.
    Hence, together with the first claim, Eq.~\eqref{eq:dimensionformula_RCP} simplifies to
    \[ n = \dim( \RC_0 (L)) +1,\]
    which completes the proof.
\end{proof}

\begin{lemma} \label{lem:charact_orthogonal_italiansq_zeroportions}
    Two Italian squares $A$ and $B$ are orthogonal if and only if 
    \[ \RC_0(A) \perp \RC_0(B).\]

\end{lemma}

\begin{proof}
Let $A$ and $B$ be Italian squares with planes $(A_i)$ and $(B_i)$, respectively. Due to the definition of orthogonal Italian squares, it follows that $A$ is orthogonal to $B$ if and only if $\langle A_i, B_j \rangle=1$ for all $i, j$. From Eq.~\eqref{eq:equivalence_innerprod_zeroportion}, we observe that
\begin{center} 
    $A$ and $B$ are orthogonal ~if and only if~ $\langle A_i^0, B_j^0 \rangle = 0$ for all $i,j$,
\end{center}
where $A_i^0$ and $B_j^0$ are the zero portions of $A_i$ and $B_j$, respectively.
\end{proof}

We are now in a position to prove our main result of this section, an upper bound on the size of a set of POIS.

\begin{theorem}
\label{thm:bound_on_POIS_size}
    There are at most $n \!-\! 1$ POIS of order~$n$.
\end{theorem}

\begin{proof}
Let $S$ be a set of POIS.  Note that for $L \in S$, the pairwise orthogonality of the matrix subspaces $\RC_0(L)$ implies that these form an orthogonal direct sum of subspaces.
We therefore obtain the result as a consequence of Lemma~\ref{lem:charact_orthogonal_italiansq_zeroportions} and Lemma~\ref{lem:dim_italiansquare_zeroportion} together with the fact that $\dim(\RC_0) = (n \!-\! 1)^2$.
\end{proof}

From the above discussion, it follows that POIS of size $n \!-\! 1$ correspond to orthogonal decompositions of $\RC_0$ into $(n \!-\! 1)$-dimensional spaces admitting a basis consisting of zero portions of ASMs.

\begin{example}
 Consider the first pair of orthogonal Italian squares from Example \ref{ex:noextend}, and denote the planes of each by $A_i$ and $B_i$. If $X\in \mathrm{RC}$, then $\langle A_i - \frac 1 4 J, X \rangle = \langle B_i - \frac 1 4 J, X \rangle = 0$ for each~$i$.
 Computing the intersection of RC with the space of solutions~$M$ to $\langle A_i, M \rangle = \langle B_i, M \rangle = 0$ for every~$i$, we obtain a space spanned by the following four matrices.
\[
\begin{pmatrix}
2 & 0 & 0 & 0 \\
1 & \!-1\! & 1 & 1 \\
\!-2\! & 0 & 2 & 2 \\
1 & 3 & \!-1\! & \!-1\!
\end{pmatrix}
,
\begin{pmatrix}
0 & 2 & 0 & 0 \\
\!-1\! & 1 & 1 & 1 \\
2 & 0 & 0 & 0 \\
1 & \!-1\! & 1 & 1
\end{pmatrix}
,
\begin{pmatrix}
0 & 0 & 2 & 0 \\
1 & 1 & 1 & \!-1\! \\
0 & 0 & 0 & 2 \\
1 & 1 & \!-1\! & 1
\end{pmatrix}
,
\begin{pmatrix}
0 & 0 & 0 & 2 \\
1 & 1 & \!-1\! & 1 \\
2 & 2 & 0 & \!-2\! \\
\!-1\! & \!-1\! & 3 & 1
\end{pmatrix}
\]
Since any ASM must have its first row being a standard basis vector, no linear combination of these matrices can be an ASM. Hence not only are the pair of Italian squares not extendable as a set of POIS, they do not even possess a common transversal. 
\end{example}

\section{Kronecker-like product for Italian squares}\label{sec:kronecker}

In this section, we define a binary operation $\ast$ for Italian squares, which generalises a well-known Kronecker-like product for Latin squares (see, for example, \cite[Th.~22.3]{VanLintWilson92}). Given two pairs $(A,A^\prime)$ and $(B,B^\prime)$ of orthogonal Italian squares of respective order $k$ and $m$, we show that $(A \ast B, A^\prime \ast B^\prime)$ is a pair of orthogonal Italian squares of order $km$, mirroring the Latin case. We then use this product to prove that there exists a pair of orthogonal Italian squares of order $km$ for all $k\ne2$ and $m \in \{5,6,7,11\}$ and sets of three POIS of order $7k$ for all $k \not\in \{2,3,6,10\}$ (none of which are Latin squares).

\begin{definition}\label{cell_addition}
    Let $C_1$ and $C_2$ be cells in Italian squares. Define $C_1 + C_2$ to be the cell containing the signed formal sum of each pair of symbols, one from $C_1$ and one from $C_2$. The sign of each summand in $C_1+C_2$ is the product of the signs of the corresponding pair of symbols in $C_1$ and $C_2$. Addition and multiplication of scalars to $C_1$ is similarly defined symbol-by-symbol.
\end{definition}

\begin{example}
    \[\begin{array}{|c|}\hline 1 \\ \hline\end{array} + 3\left(\:\begin{array}{|c|}\hline \SmallArray{1-2\\+3} \\ \hline\end{array} - 1\right)
    = \begin{array}{|c|}\hline 1 \\ \hline\end{array} + 3\left(\:\begin{array}{|c|}\hline \SmallArray{0-1\\+2} \\ \hline\end{array}\:\right)
    = \begin{array}{|c|}\hline 1 \\ \hline\end{array} + \begin{array}{|c|}\hline \SmallArray{0-3\\+6} \\ \hline\end{array}
    = \begin{array}{|c|}\hline \SmallArray{1-4\\+7} \\ \hline\end{array}\]
    \[\begin{array}{|c|}\hline \SmallArray{1-2\\+3} \\ \hline\end{array} + 3\left(\:\begin{array}{|c|}\hline \SmallArray{1-2\\+3} \\ \hline\end{array} - 1\right)
    = \begin{array}{|c|}\hline \SmallArray{1-2\\+3} \\ \hline\end{array} + 3\left(\:\begin{array}{|c|}\hline \SmallArray{0-1\\+2} \\ \hline\end{array}\:\right)
    = \begin{array}{|c|}\hline \SmallArray{1-2\\+3} \\ \hline\end{array} + \begin{array}{|c|}\hline \SmallArray{0-3\\+6} \\ \hline\end{array}
    = \begin{array}{|c|}\hline \SmallArray{1-2+3\\-4+5-6\\+7-8+9} \\ \hline\end{array}\]
\end{example}

Given an $m \times m$ matrix $A=(a_{i,j})$ and an $n \times n$ matrix $B$, recall the \emph{Kronecker product} \cite[p.~201]{VanLintWilson92} of $A$ and $B$, defined to be the following $mn \times mn$ matrix.
\[A \otimes B = \begin{pmatrix}
    a_{1,1}B & a_{1,2}B & \dots & a_{1,m}B \\
    a_{2,1}B & a_{2,2}B & \ddots & \vdots \\
    \vdots & \ddots & \ddots & a_{m-1,m}B \\
    a_{m,1}B & \dots & a_{m,m-1}B & a_{m,m}B
\end{pmatrix}\]

\begin{definition}
Given an $m \times m$ array $A$ and an $n \times n$ array $B$, define $A*B$ to be the $mn \times mn$ array given by 
\[
A*B = (A-J_m)\otimes (nJ_n) + J_m\otimes B,
\]
where $\otimes$ is the standard Kronecker product, addition and subtraction of cells is defined as in Definition \ref{cell_addition}, and $J_i$ is the all-ones matrix of size $i$.
\end{definition}

\begin{example}
\[
\begin{array}{|c|c|c|}
\hline
    1 &2& 3 \\
\hline
     2&\SmallArray{1-2\\+3}&2 \\
\hline
     3&2&1 \\
\hline
\end{array}
*
\begin{array}{|c|c|}
\hline
    1 &2 \\
\hline
     2&1 \\
\hline
\end{array}
=
\begin{array}{|c|c||c|c||c|c|}
\hline
  1  & 2      &  3 &  4     &  5 &  6 \\
\hline
  2  & 1      &  4 &  3     &  6 &  5 \\
\hline\hline
  3  &  4     & \SmallArray{1-3\\+5}  &   \SmallArray{2-4\\+6}    &  3 & 4  \\
\hline
  4  &  3     &  \SmallArray{2-4\\+6} &   \SmallArray{1-3\\+5}    &  4 &  3 \\
\hline\hline
  5  &  6     & 3  &  4     &  1 & 2  \\
\hline
  6  &  5     &  4 &  3     &  2 & 1  \\
\hline
\end{array}
\]
\[
\begin{array}{|c|c|}
\hline
    1 &2 \\
\hline
     2&1 \\
\hline
\end{array}
*
\begin{array}{|c|c|c|}
\hline
    1 &2& 3 \\
\hline
     2&\SmallArray{1-2\\+3}&2 \\
\hline
     3&2&1 \\
\hline
\end{array}
=
\begin{array}{|c|c|c||c|c|c|}
\hline
1&2   &3       &                            4 &  5 &6      \\
\hline
2& \SmallArray{1 - 2\\ +3}  &  2     &      5 &  \SmallArray{ 4- 5\\ +6} &  5   \\
\hline
3& 2  & 1      &                              6 &5   & 4     \\
\hline
\hline
4&  5 & 6      &   1 &2   & 3  \\
\hline
5& \SmallArray{ 4- 5\\ +6}  &   5 & 2 &  \SmallArray{1 - 2\\ +3} & 2  \\
\hline
6& 5  &  4     &    3 &  2 & 1  \\
\hline
\end{array}
\]
\[
\begin{array}{|c|c|c|}
\hline
    1 &2& 3 \\
\hline
     2&\SmallArray{1-2\\+3}&2 \\
\hline
     3&2&1 \\
\hline
\end{array}
*
\begin{array}{|c|c|c|}
\hline
    1 &2& 3 \\
\hline
     2&\SmallArray{1-2\\+3}&2 \\
\hline
     3&2&1 \\
\hline
\end{array}
=
\begin{array}{|c|c|c||c|c|c||c|c|c|}
\hline
1&2   &3       &                            4 &  5 &6       &      7 &   8&9   \\
\hline
2& \SmallArray{1 - 2\\ +3}  &  2     &      5 &  \SmallArray{ 4- 5\\ +6} &  5     &       8& \SmallArray{ 7- 8\\ +9}  &  8 \\
\hline
3& 2  & 1      &                              6 &5   & 4      &     9  &  8 & 7  \\
\hline
\hline
4&  5 & 6      &      \SmallArray{1-4\\+7} &  \SmallArray{2-5\\+8} &  \SmallArray{3-6\\+9}     &       4&  5 &6   \\
\hline
5& \SmallArray{ 4- 5\\ +6}  &  5     &    \SmallArray{2-5\\+8}   &  \SmallArray{1-2+3\\-4+5-6\\+7-8+9} &  \SmallArray{2-5\\+8}     &      5 & \SmallArray{ 4- 5\\ +6}  & 5  \\
\hline
6& 5  &  4     &     \SmallArray{3-6\\+9}  &  \SmallArray{2-5\\+8} &    \SmallArray{1-4\\+7}   &      6 & 5  & 4  \\
\hline
\hline
7& 8  &  9     &      4 & 5  &  6     &      1 &2   & 3  \\
\hline
8&  \SmallArray{ 7- 8\\ +9} &   8    &      5 &  \SmallArray{ 4- 5\\ +6} &  5     &      2 &  \SmallArray{1 - 2\\ +3} & 2  \\
\hline
9&  8 &   7    &    6   & 5  &   4    &      3 &  2 & 1  \\
\hline
\end{array}
\]
\end{example}

It is known that for MOLS $\{A_1,A_2,\dots,A_k\}$ and $\{B_1,B_2,\dots,B_k\}$ of order $m$ and $n$, respectively, $\{A_1* B_1,A_2* B_2,\dots,A_k* B_k\}$ is a set of $mn \times mn$ MOLS \cite[Th.~22.3]{VanLintWilson92}. We now prove the equivalent result for Italian squares.

\begin{lemma}\label{lem: Kron IS}
    Let $A$ be an Italian square of order $m$ and $B$ be an Italian square of order $n$. Then $A * B$ is an Italian square.
\end{lemma}

\begin{proof}
    As Italian squares are in bijection with ASHMs, we prove the equivalent result for the ASHMs $X$ and $Y$, respectively corresponding to $A$ and $B$. It follows from the definition of $*$ that the Italian square resulting from $A * B$ corresponds to an $mn \times mn \times mn$ hypermatrix $Z$ with the following entries, for $i_1,j_1,k_1\in[m]$ and $i_2,j_2,k_2\in[n]$.
    \[Z_{(i_1-1)n+i_2, (j_1-1)n+j_2,(k_1-1)n+k_2} = X_{i_1,j_1,k_1}Y_{i_2,j_2,k_2}\]
    For fixed $j_1,j_2,k_1,k_2$, consider moving along all values of $(i_1,i_2)\in[m]\times[n]$. All values in this column are zero until the first position where $X_{i_1,j_1,k_1},Y_{i_2,j_2,k_2}$ are both non-zero. Since $X$ and $Y$ are ASHMs, we have $X_{i_1,j_1,k_1} = Y_{i_2,j_2,k_2} = 1$ and therefore the first non-zero in the column is $+1$. The non-zero entries in the column then alternate in sign following the pattern of the corresponding column of $Y$. In particular, the last before $i_1$ increases is $+1$. If there are any more non-zero entries in the corresponding $X$ column, the next will be $-1$. Therefore the non-zeros continue to alternate in sign, this time beginning and ending with $-1$. This continues until the final non-zero in the corresponding $X$ column, which again is $+1$. Therefore the non-zero entries in the entire column alternate in sign, beginning and ending with $+1$. The same is true for rows and vertical lines. Therefore $Z$ is an ASHM, and $A * B$ is an Italian square.
\end{proof}

\begin{lemma}\label{lem: Kron ortho}
    Let $A, A'$ be orthogonal Italian squares of order $m$, and $B, B'$ orthogonal Italian squares of order $n$. Then $A * B$ is orthogonal to $A' * B'$.
\end{lemma}

\begin{proof}
    We again prove the equivalent result for ASHMs $X$, $X'$, $Y$, $Y'$, $Z$, $Z'$ respectively corresponding to Italian squares $A$, $A'$, $B$, $B'$, $A * B$, $A' * B'$. For any layers $Z_i$ of $Z$ and $Z'_j$ of $Z'$ we show that $\langle Z_i, Z'_j \rangle = 1$. Indeed, writing $i = (i_1 - 1)n + i_2$ and $j = (j_1 - 1)n + j_2$, it follows that $Z_i = X_{i_1} * Y_{i_2}$ and $Z'_j = X'_{j_1} * Y'_{j_2}$, where $X_{i_1}, X'_{j_1}$ and $Y_{i_2}, Y'_{j_2}$ are layers of $X, X'$ and $Y, Y'$, respectively. We conclude that $\langle Z_i, Z'_j \rangle = \langle X_{i_1} * Y_{i_2}, X'_{j_1} * X'_{j_2} \rangle = \langle X_{i_1}, X'_{j_1} \rangle \langle Y_{i_2}, Y'_{j_2} \rangle = 1$, since $\langle X_{i_1}, X'_{j_1} \rangle = 1 = \langle Y_{i_2}, Y'_{j_2} \rangle$ by orthogonality of $A, A'$ and $B, B'$. 
\end{proof}

The following result is a consequence of Lemmas~\ref{lem: Kron IS} and~\ref{lem: Kron ortho}, along with examples of POIS of small order.

\begin{theorem}\label{thm:explicit}
    There exists a pair of POIS of order $km$, neither of which is a Latin square, for $k\ne2$ and $m \in \{5,6,7,11\}$. Moreover, there exists a set of three POIS of order $7k$, none of which is a Latin square, for all $k\not\in\{2,3,6,10\}$.
\end{theorem}

\begin{proof}
    Example \ref{6x6} gives us a pair of POIS of order $k=6$. For all other $k \ge 3$, there exist a pair of MOLS of order $k$. We construct a pair of POIS $A_1 * B_1$ and $A_2 * B_2$ of the desired order, where $B_1$ and $B_2$ have order $k$, and $A_1$ and $A_2$ are the following Italian squares of order $m$.

    \underline{$m = 5$:}
    \[A_1 = \begin{array}{|c|c|c|c|c|}
            \hline
            1&3&4&2&5\\
            \hline
            2&1&3&5&4\\
            \hline
            3&\SmallArray{2-3\\+5}&1&4&3\\
            \hline
            4&3&\SmallArray{2-3\\+5}&3&1\\
            \hline
            5&4&3&1&2\\
            \hline
        \end{array} ~, \quad A_2 = \begin{array}{|c|c|c|c|c|}
            \hline
            4&1&2&5&3\\
            \hline
            3&2&4&1&5\\
            \hline
            2&4&\SmallArray{1-2+3\\-4+5}&4&2\\
            \hline
            1&5&2&3&4\\
            \hline
            5&3&4&2&1\\
            \hline
        \end{array}\]
        
    \underline{$m = 6$:} $A_1$ and $A_2$ of Example \ref{6x6}.

    \underline{$m = 7$:} $\{\mathcal D_7, A_1, A_2\}$ form a set of 3 POIS. There exists a set of 3 MOLS of order $k$ for all $k\not\in\{1,2,3,6,10\}$ \cite[Table 3.88]{ColbournDinitz06}. The result for three POIS follows.
    \[A_1 = \begin{array}{|c|c|c|c|c|c|c|}
            \hline
            1&4&3&2&5&6&7\\
            \hline
            5&3&7&1&6&2&4\\
            \hline
            4&\SmallArray{2-4\\+6}&4&7&1&5&3\\
            \hline
            6&7&2&5&4&3&1\\
            \hline
            7&5&6&3&2&\SmallArray{1-2\\+4}&2\\
            \hline
            2&4&1&6&3&7&5\\
            \hline
            3&1&5&4&7&2&6\\
            \hline
        \end{array} ~, \quad
        A_2 = \begin{array}{|c|c|c|c|c|c|c|}
            \hline
            1&2&3&6&5&4&7\\
            \hline
            4&6&2&7&1&5&3\\
            \hline
            5&3&7&1&4&\SmallArray{2-4\\+6}&4\\
            \hline
            7&5&4&3&6&1&2\\
            \hline
            6&\SmallArray{4-6\\+7}&6&5&2&3&1\\
            \hline
            3&1&5&2&7&4&6\\
            \hline
            2&6&1&4&3&7&5\\
            \hline
        \end{array}\]

        \underline{$m = 11$:} $A_1 = \mathcal D_{11}$ and
    \[A_2 = \begin{array}{|c|c|c|c|c|c|c|c|c|c|c|}
    \hline
            1 & 2 & 3 & 4 & 5 & 6 & 7 & 8 & 9 & 10 & 11\\ \hline
            2 & 3 & \SmallArray{1-3\\+4} & \SmallArray{3-4\\+5} & \SmallArray{4-5\\+8} & 5 & 10 & \SmallArray{6-8\\+11} & 8 & 7 & 9\\ \hline
            8 & 7 & 9 & 10 & 6 & 1 & \SmallArray{5-10\\+11} & 2 & \SmallArray{3-8\\+10} & 4 & 8\\ \hline
            7 & \SmallArray{6-7\\+10} & 8 & 11 & \SmallArray{5-8\\+9} & \SmallArray{4-5\\+7} & 8 & 1 & 5 & 3 & 2\\ \hline
            9 & 7 & 11 & 6 & \SmallArray{8-9\\+10} & 5 & \SmallArray{4-7\\+9} & 3 & \SmallArray{2-3\\+7} & 1 & 3\\ \hline
            11 & 9 & 5 & 4 & 3 & 2 & 10 & 7 & 6 & 8 & 1\\ \hline
            4 & 11 & \SmallArray{3-4\\+10} & 9 & 7 & 8 & 6 & 4 & 1 & \SmallArray{2-4\\+5} & 4\\ \hline
            10 & 5 & \SmallArray{2-5\\+7} & 8 & \SmallArray{1-8\\+11} & 9 & 3 & 5 & 8 & 4 & 6\\ \hline
            3 & 4 & 5 & 7 & \SmallArray{2-7\\+8} & 11 & \SmallArray{1-4\\+7} & 9 & 4 & 6 & 10\\ \hline
            5 & 8 & 6 & 1 & 7 & \SmallArray{3-8\\+10} & \SmallArray{2-3\\+4} & 8 & \SmallArray{3-5\\+11} & 9 & 5\\ \hline
            6 & 1 & 4 & 2 & 9 & 8 & 3 & 10 & 5 & 11 & 7\\ \hline
        \end{array} ~. \]
\end{proof}

\section{Conclusion and future work}\label{sec:conclude}

In this paper we investigated a natural generalisation of Latin squares in which permutation
matrices are replaced with alternating sign matrices (ASMs). Building on the identification
between Latin squares and permutation hypermatrices, we studied alternating sign hypermatrices, in which
each planar slice of all three types is an ASM. We introduced the notation of \emph{Italian
squares} as a compressed representation of such hypermatrices, better suited to the study of
orthogonality than the Latin-like squares of Brualdi and Dahl. Within this framework we addressed several natural problems: orthogonality at the
local level (orthogonal mates for a single ASM in \Cref{sec:mate}, and existence of transversals of an Italian
square in \Cref{sec:transversals}), orthogonality at the global level (pairwise orthogonality between Italian squares, and the maximum size of such a set in \Cref{sec:pois}), systematic constructions of pairs of orthogonal Italian squares
via a Kronecker-like product in \Cref{sec:kronecker}, and partial structure completion through an analogue of the
Frobenius-K\"onig theorem in \Cref{sec:frob}. These results also resolve several open questions from
Brualdi and Dahl~\cite{BrualdiDahl18,BrualdiDahl24}.

Several natural questions arise from this work, which we now discuss. \medskip

\noindent\textbf{Constructions of POIS.}
In the case of Latin squares, it is well-known that for each $n > 1$, with the exception of $n = 2, 6$, there exist a pair of orthogonal Latin squares of order $n$. 
One of the original motivations for studying Italian squares is the suggestion of Brualdi and Dahl~\cite{BrualdiDahl18} that they (under the alternating sign hypermatrix name) might allow the bound on the number of MOLS to be exceeded.
Our bound of $n \!-\! 1$ of \Cref{thm:bound_on_POIS_size} matches the Latin square general bound, but this does not preclude Italian squares from breaking records for specific values of $n$. 
The case $n = 6$ already provides one such example: a pair of POIS of order~$6$ exists
despite the classical impossibility of a pair of MOLS of that order. On the other hand, exhaustive computer search indicates that a pair of orthogonal Italian squares, neither of which is Latin, can only exist for $n \ge 5$. 
Using the Kronecker-like product of Section~\ref{sec:kronecker}, we constructed such pairs for orders $5k$, $6k$, $7k$, and $11k$ for all $k \neq 2$. 
Extending this to all orders will require additional constructions, and it would be particularly interesting to find analogues of the finite-field constructions that underpin the theory of MOLS.

\vspace{-2\baselineskip}
\begin{quote}
\begin{problem}
Does there exist a pair of orthogonal Italian squares, not both Latin (resp.\ neither Latin), for each $n \ge 5$?
\end{problem}
\end{quote}

The first open case where the bound on MOLS could be strictly exceeded by POIS is $n = 10$, where it is unknown whether there exist three POIS. It would also be of interest to determine, for non-prime-power~$n$, whether $n \!-\! 1$ POIS always exist at all.
\vspace{-2\baselineskip}
\begin{quote}
\begin{problem}
Are there $n\in\mathbb{N}$ for which the maximum number of POIS exceeds the maximum number of MOLS of order $n$? Moreover, are there non-prime-power $n$ with a set of $n \!-\! 1$ POIS of order~$n$?
\end{problem}
\end{quote}

\medskip
\noindent\textbf{The diamond Italian square.}
We have shown that the diamond Italian square $\mathcal D_n$ admits a
transversal if and only if $n \equiv \pm 1 \bmod 6$.  
It remains open whether there are at least $n$ such transversals, and whether $n$ of them can be arranged to form an Italian square. 
If this is the case, and at least one transversal is not a permutation matrix, this would provide a systematic construction of a set of two POIS, neither of which is a Latin square.

\medskip
\noindent\textbf{Existence problems in~$\W_n$.}
Section \ref{sec:frob} introduces the set $\W_n$ consisting of all $(0,\pm1)$-matrices in which the non-zero entries of each row and column alternate in sign, and the sum of each row/column is in $\{0,\pm1\}$. In addition to ASMs of order $n$ themselves, $\W_n$ contains toroidal ASMs of order $n$ and other generalisations. The \emph{Gale-Ryser Theorem} \cite{Gale57, Ryser57} gives an existence condition for $(0,1)$-matrices with specified row and column sums. One could consider the following generalisation.
\vspace{-2\baselineskip}
\begin{quote}
\begin{problem} For which $(0,\pm1)$-vectors $R$ and $S$ of order $n$ does there exist $X \in \W_n$ with row-sums $R$ and column-sums $S$?
\end{problem}
\end{quote}

\medskip
\noindent\textbf{Linear-algebraic methods.}
The zero-portion perspective of Section~\ref{sec:transversals} reduces orthogonality
of Italian squares to the geometric condition $\RC_0(P) \perp \RC_0(Q)$, and reduces the
existence of a transversal to a linear feasibility problem over the integers. We expect that
these linear-algebraic methods can be developed further into a more systematic toolkit, for
instance by using the structure of the solution spaces to guide both theoretical analysis and
computational search.

\medskip
\noindent\textbf{Computational aspects.}
Several results in this paper, including the characterisation in Example~\ref{ex:noextend},
relied on exhaustive computer search. Improving the efficiency of such searches, and extending
them to larger orders, would help clarify the landscape of POIS and inform future results and conjectures.

Our current state of knowledge is summarised by the table below, which for order~$n$ gives the largest size of a set of POIS containing exactly~$i$ Italian squares which are not Latin squares.  We use a simple enumeration technique which generates all ASM planes of orthogonal Italian squares consecutively with an early-abort strategy.

\begin{table}[h]\centering

\begin{tabular}{|c||c|c|c|c|c|}\hline
$n \backslash i$ & 0 & 1 & 2 & 3 & 4 \\\hline
4 & 3 & 2 & 0 & 0 & - \\
5 & 4 & 3 & 2 & 0 & 0 \\
6 & 1 & 2 & 2 & ? & ? \\
7 & 6 & $\ge 2$ & $\ge 2$ & $\ge 3$ & ? \\\hline
\end{tabular}
\end{table}

For example, if $n = 6$ no set of three POIS containing a Latin square exists, but we do not know yet whether there exist any set of three POIS.

\section*{Acknowledgment}

The authors thank the organisers of the “1st Dublin Discrete Mathematics Workshop” for providing a stimulating environment in which this collaboration began. S.\,L.\ was supported by Kempe Foundation, grant JCSMK24-01105.

{\small \bibliographystyle{plain}
\bibliography{bibliographypois}}

\footnotesize

\noindent Alena Ernst, \textsc{Department of Mathematical Sciences, Worcester Polytechnic Institute, Worcester, MA, USA} \\
\noindent \textit{Email address}: \texttt{aernst@wpi.edu}\medskip

\noindent Stefano Lia, \textsc{Department of Mathematics and Mathematical Statistics, Umeå University, Umeå, Sweden} \\
\noindent \textit{Email address}: \texttt{stefano.lia@umu.se}\medskip

\noindent Cian O'Brien, \textsc{Department of Mathematics and Computer Studies, Mary Immaculate College, Limerick, Ireland} \\
\noindent \textit{Email address}: \texttt{obrien.cian@outlook.ie}\medskip

\noindent John Sheekey, \textsc{School of Mathematics and Statistics, University College Dublin, Dublin, Ireland} \\
\noindent \textit{Email address}: \texttt{john.sheekey@ucd.ie}\medskip

\noindent Jens Zumbrägel, \textsc{Faculty of Computer Science and Mathematics, University of Passau, Passau, Germany} \\
\noindent \textit{Email address}: \texttt{jens.zumbraegel@uni-passau.de}

\end{document}